\RequirePackage{fix-cm}
\documentclass[smallextended]{svjour3}       
\smartqed  
\usepackage{graphicx}
\usepackage{dsfont,latexsym}
\usepackage{mathrsfs} 

\usepackage{amssymb,amscd,amsxtra} 

\usepackage{color}
\definecolor{citation}{rgb}{0.2,0.58,0.2} 
\definecolor{formula}{rgb}{0.1,0.2,0.6}
\definecolor{url}{rgb}{0.3,0,0.5}  

\usepackage[colorlinks=true,linkcolor=formula,urlcolor=url,citecolor=citation]{hyperref}  
  
\newcommand{\R}{{\mathds R}}
\newcommand{\n}{{\mathds n}}
 
\DeclareMathOperator*{\esssup}{ess\,sup}
\DeclareMathOperator*{\essinf}{ess\,inf}

\DeclareMathOperator*{\osc}{osc}

\DeclareMathOperator*{\dist}{dist}

   \def\XXint#1#2#3{{\setbox0=\hbox{$#1{#2#3}{\int}$}
        \vcenter{\hbox{$#2#3$}}\kern-.5\wd0}}

\def\dxy{\,{\rm d}x{\rm d}y}
\def\dx{\,{\rm d}x}
\def\dy{\,{\rm d}y}
\def\dt{\,{\rm d}t}

\def\vs{\vspace{0.9mm}}

\def\eps{\varepsilon}

\allowdisplaybreaks
\makeatletter
\DeclareRobustCommand*{\bfseries}{%
  \not@math@alphabet\bfseries\mathbf
  \fontseries\bfdefault\selectfont
  \boldmath
}
\makeatother

\def\mean#1{\mathchoice%
          {\mathop{\kern 0.2em\vrule width 0.6em height 0.69678ex depth -0.58065ex
                  \kern -0.8em \intop}\nolimits_{\kern -0.4em#1}}%
          {\mathop{\kern 0.1em\vrule width 0.5em height 0.69678ex depth -0.60387ex
                  \kern -0.6em \intop}\nolimits_{#1}}%
          {\mathop{\kern 0.1em\vrule width 0.5em height 0.69678ex
              depth -0.60387ex
                  \kern -0.6em \intop}\nolimits_{#1}}%
          {\mathop{\kern 0.1em\vrule width 0.5em height 0.69678ex depth -0.60387ex
                  \kern -0.6em \intop}\nolimits_{#1}}}
           
\newlength{\defbaselineskip}
\newcommand{\setlinespacing}[1]
           {\setlength{\baselineskip}{#1 \defbaselineskip}}

\begin{document}

\title{The obstacle problem for nonlinear integro-differential operators\thanks{The first author has been supported by the Magnus Ehrnrooth Foundation (grant no. ma2014n1, ma2015n3). The second author has been supported by the Academy of Finland. The third author is a member of Gruppo Nazionale per l'Analisi Matematica, la Probabilit\`a e le loro Applicazioni (GNAMPA) of Istituto Nazionale di Alta Matematica ``F.~Severi'' (INdAM), whose support is acknowledged.
\\ The results of this paper have been announced in the preliminary research report~\cite{KKP16}.}
}


\author{Janne Korvenp\"a\"a       \and
        Tuomo Kuusi 
        \and Giampiero Palatucci
}

\authorrunning{J. Korvenp\"a\"a, T. Kuusi, G. Palatucci} 

\institute{J. Korvenp\"a\"a, T. Kuusi \at
             Department of Mathematics and Systems Analysis, Aalto University\\
             P.O. Box 1100\\
             00076 Aalto, Finland\\
              Telefax:  +358 9 863 2048\\
              \email{janne.korvenpaa@aalto.fi} \  \email{tuomo.kuusi@aalto.fi}         
           \and
           G. Palatucci \at
              Dipartimento di Matematica e Informatica, Universit\`a degli Studi di Parma\\
              Campus - Parco Area delle Scienze,~53/A\\
              43124 Parma, Italy\\
              Tel: +39 521 90 21 11\\ 
              \email{giampiero.palatucci@unipr.it}
}

\date{      }

\maketitle

\vspace{-1.9cm}

\noindent
\\ to appear in {\large \it Calc. Var. Partial Differential Equations}

{\vspace{6mm} \it \small Qualche tempo dopo Stampacchia, partendo sempre dalla sua disequazione variazionale, 
 aperse un nuovo campo di ricerche che si rivel\`o importante e fecondo. 
 Si tratta di quello che oggi \`e chiamato {\rm ``}il problema 
dell'ostacolo{\rm''}. \hfill {\rm [}Sandro Faedo, 1987\,{\rm]}}\vspace{10mm}

\begin{abstract}
We investigate the obstacle problem for a class of nonlinear equations driven by nonlocal, possibly degenerate, integro-differential operators, whose model is the fractional $p$-Laplacian operator with measurable coefficients. 
Amongst other results, we will prove both the existence and uniqueness of the solutions to the obstacle problem, and that these solutions inherit regularity properties, such as boundedness, continuity and H\"older continuity (up to the boundary),  from the obstacle.
\end{abstract}

\tableofcontents

\setlinespacing{1.08}

\section{Introduction}

In the present paper, we study the obstacle problem related to the following nonlocal and nonlinear operator, defined formally as
\begin{equation}\label{operatore}
\mathcal{L}u(x)=p.~\!v. \int_{\R^n} K(x,y)|u(x)-u(y)|^{p-2}\big(u(x)-u(y)\big)\dy, \qquad x\in \R^n;
\end{equation}
we take differentiability of order $s\in (0,1)$ and growth $p>1$. The kernel $K$ is of order $(s,p)$ (see~\eqref{hp_k}) with merely measurable coefficients.  
The integral in~\eqref{operatore} may be singular at the origin and must be interpreted in an appropriate sense. Since we assume that coefficients are merely measurable, the involved equation has to have a suitable weak formulation; see Section~\ref{sec_preliminaries} below for the precise assumptions on the involved quantities.
 
\smallskip
 
The  obstacle problem involving fractional powers of the Laplacian operator naturally appears in many contexts, such as in the analysis of anomalous diffusion (\cite{BG90}), in the so called quasi-geostrophic flow problem (\cite{CV10}), and in pricing of American options regulated by assets evolving in relation to jump processes (\cite{CT04}). 
In particular, the last application made the obstacle problem very relevant in recent times in all its forms; the obstacle problem can be indeed stated in several ways.
Roughly speaking, a solution $u$ to the fractional obstacle problem is a minimal weak supersolution to the equation
\begin{equation}\label{equazione}
\mathcal{L}u=0
\end{equation} 
above an obstacle function $h$.

\smallskip

In the linear case when $p=2$ and when the kernel~$K$ reduces to the Gagliardo kernel~$K(x,y)=|x-y|^{-n-2s}$ without coefficients, a large treatment of the fractional obstacle problem can be found for instance in the fundamental papers by Caffarelli, Figalli, Salsa, and Silvestre (see, e.~\!g.,~\cite{Sil07,CSS08,CF13} and the references therein). See also~\cite{Foc10,Foc09} for the analysis of families of bilateral obstacle problems involving fractional type energies in aperiodic settings; 
the paper~\cite{PP15b} for the fractional obstacle problems with drift; and the recent papers~\cite{Gua15,MN15} for related estimates and approximations results.
This topic, despite its relatively short history, has already evolved into quite an elaborate theory, with connections to numerous branches of Analysis. It is impossible to provide here a complete list of references. We refer the interested reader to the exhaustive recent lecture notes by Salsa (\cite{Sal12}), for the obstacle problem in the pure fractional Laplacian case, with the natural connection to the thin obstacle problem in low dimensions (for which we refer to~\cite{ACS08}).

\smallskip

However, in the more general framework considered here, the panorama seems rather incomplete, if not completely deficient in results. Clearly, the main difficulty into the treatment of the operators~$\mathcal{L}$ in~\eqref{operatore} lies in their very definition, which combines the typical issues given by its nonlocal feature together with the ones given by its nonlinear growth behavior; also, further efforts are needed due to the presence of merely measurable coefficients in the kernel~$K$. For this, some very important tools recently introduced in the nonlocal theory, as the by-now classic $s$-harmonic extension (\cite{CS07}), the strong three-term commutators estimates (\cite{DLR11}), and other  successful tricks as e.~\!g. the pseudo-differential commutator approach in~\cite{PP14,PP15}, cannot be plainly applied and seem difficult to adapt to the nonlinear framework considered here (mainly due to the non-Hilbertian nature of the involved fractional Sobolev spaces $W^{s,p}$).

\smallskip

Nevertheless, some related regularity results have been very recently achieved in this context, in~\cite{BL15,DKP14,DKKP15,DKP15,KMS15,KMS15b,IMS15,IS14,Sch15} and many others, where often a fundamental role to understand the nonlocality of the nonlinear operators~$\mathcal{L}$ has been played by a special quantity,
\begin{equation}\label{coda}
{\rm Tail}(u;x_0,r) := \bigg(r^{sp} \int_{\R^n \setminus B_r(x_0)} |u(x)|^{p-1} |x-x_0|^{-n-sp} \dx \bigg)^{\frac{1}{p-1}};
\end{equation}
that is, {\it the nonlocal tail} of a function $u$ in the ball of radius $r>0$ centered in $x_0 \in \R^n$. This quantity, introduced by two of the authors with A. Di Castro in~\cite{DKP15}, have been subsequently became a relevant factor in many instances when one requires a fine quantitative control of the long-range interactions, which naturally arise when dealing with nonlocal operators (see Section~\ref{sec_preliminaries} below).

\smallskip

For what concerns the main topic in the present paper, i.~\!e., the nonlinear fractional obstacle problem with coefficients, we will prove a series of both qualitative and quantitative results. Amongst them, 
we will formulate the natural variational framework for the obstacle problem, and we will prove both the existence and uniqueness of the solution~$u$ to this variational formulation (Theorem~\ref{obst prob sol}). We will show that such a solution is a weak supersolution and that it is the smallest supersolution above the obstacle in a suitable sense (Proposition~\ref{smallest super}).
We will also demonstrate that the solution~$u$ inherits the regularity of the obstacle, namely the boundedness~(Theorem~\ref{thm:obs bnd}), continuity (Theorem~\ref{thm:obs cont}), and H\"older continuity (Theorem~\ref{thm:obs H cont}). As a consequence,  assuming that the obstacle function~$h$ is continuous,~$u$ is a weak solution to~\eqref{equazione} in the open set $\{u>h\}$ (Corollary~\ref{obst prob free}). These results are in clear accordance with the aforementioned results for the obstacle problems in the pure fractional Laplacian $(-\Delta)^s$ case. However, our approach here is different and, though we are dealing with a wider class of nonlinear integro-differential operators with coefficients, the proofs are even somehow simpler, since we can make effort of a new nonlocal set-up together with the recent quantitative estimates obtained in~\cite{DKP14,DKP15,DKKP15}, by also extending to the fractional framework some important tools in the classical Nonlinear Potential Theory.

\smallskip

Finally, we will deal with the regularity up to the boundary (Theorems~\ref{thm:H cont bdry}-\ref{thm:cont bdry}). As well known, in the contrary with respect to the interior case, boundary regularity for nonlocal equations driven by singular, possibly degenerate, operators as in~\eqref{operatore} seems to be a difficult problem in a general {\it nonlinear} framework under natural assumptions on the involved quantities (while we refer to the recent paper~\cite{Ros14} and to the forthcoming survey~\cite{Ros16} for the case $p=2$). In this respect, a first (and possibly the solely currently present in the literature) result of global H\"older regularity has been obtained very recently in the  interesting paper~\cite{IMS15}, where the authors deal with the equation in~\eqref{equazione}, in the special case when the operator $\mathcal{L}$ in~\eqref{operatore} coincides with the nonlinear fractional Laplacian~$(-\Delta)^{s}_p$, by considering exclusively zero Dirichlet boundary data, and by assuming a strong $C^{1,1}$-regularity up to the boundary for the domain~$\Omega$. Indeed, their proof is strongly based on the construction of suitable barriers near $\partial \Omega$, starting from the fact that, under their restrictive assumptions, the function $x\mapsto x^s_+$ is an explicit solution in the half-space. Clearly, one cannot expect to plainly extend such a strategy in the general framework considered here, in view of the presence of merely measurable coefficients in~\eqref{operatore}. Also, we will allow nonzero boundary Dirichlet data to be chosen, and we will assume the domain~$\Omega$ only to satisfy a natural {\it measure density condition} (precisely, just on the complement of~$\Omega$; see Formula~\eqref{eq:dens cond} on Page~\pageref{eq:dens cond}); the latter being as expected in accordance with the classical Nonlinear Potential Theory (that is, when $s=1$).
For this, we will need a new proof, that will extend up to the boundary part of the results in~\cite{DKP14,DKP15} together with a careful handling of the tail-type contributions (see Section~\ref{sec_boundary}). Once again, it is worth stressing that all these results are new even in the pure fractional $p$-Laplacian case when the operator~$\mathcal{L}$ does coincide with $(-\Delta)^s_p$, and  in the case of the (linear) fractional Laplacian with coefficients.

\smallskip

All in all, let us summarize  
 the contributions of the present paper:  
\noindent
\\ - We prove new regularity results in terms of boundedness, continuity, and H\"older continuity for the solutions to a very general class of nonlocal obstacle problems, by extending previous results in the literature valid only for the pure linear fractional Laplacian case $(-\Delta)^s$ without coefficients, also giving new proofs even in that case;
\noindent
\\ - We obtain new regularity results up to the boundary for nonlocal operators, and, since we allow the obstacle function~$h$ to be an extended real-valued function, the degenerate case when $h\equiv -\infty$ (i.~\!e., no obstacle is present) does reduce the problem to the standard Dirichlet boundary value problem, so that the results proven here are new even when $\mathcal{L}$ does coincide with the fractional $p$-Laplacian $(-\Delta)^s_p$. Also, since we assume that the boundary data merely belong to an appropriate tail space~$L^{p-1}_{sp}(\R^n)$, all the (inner and boundary) results here reveal to be an improvement with respect to the previous ones in the literature when the data are usually given in the whole fractional Sobolev space~$W^{s,p}(\R^n)$;
\noindent
\\ - By solving the fractional obstacle problem together with some of the expected basic results proven here, we provide an important tool for several further investigations and applications.
Indeed, as well known, the obstacle problem is deeply related to the study of minimal surfaces and the capacity of a set in Nonlinear Potential Theory. Thus, by means of our framework, we possibly give the basis for the development of a {\it nonlocal} Nonlinear Potential Theory. This can be already seen in some subsequent forthcoming papers, as, e.~\!g., in ~\cite{KKL16} where part of the results here are the key for the viscosity approach for nonlocal integro-differential operators, and in~\cite{DKKP15} where the whole nonlocal obstacle set-up is needed to extend the classical Perron method to a nonlocal nonlinear setting.

\smallskip

Finally, let us comment about some immediate open problems naturally arising in this framework. Firstly, one can argue about the optimal regularity for the solutions to the nonlocal obstacle problem. We recall that for the classical obstacle problem, when $\mathcal{L}$ coincides with the Laplacian operator $-\Delta$,  the solutions are known to belong to~$C^{1,1}$. 
The intuition behind this regularity result goes as follows:  in the contact set one has $-\Delta u = -\Delta h$, while where $u >h$ one has $-\Delta u=0$; since the Laplacian jumps from $-\Delta h$ to $0$ across the free boundary, the second derivatives of $u$ must have a discontinuity, so that $C^{1,1}$ is the maximum regularity class that can be expected. In the contrary, when $\mathcal{L}\equiv(-\Delta)^s$, despite
the previous local argument does suggest that the solutions $u$ belong to $C^{2s}$,  the optimal regularity is $C^{1,s}$, and this is quite surprising since the regularity exponent is higher than the order of the equation.  
In the general nonlocal framework, starting from the H\"older regularity proven here, we still expect higher regularity results as for the linear case; nevertheless, in view of the interplay between local and nonlocal contributions and without having the possibility to rely on the $s$-harmonic extension, it is not completely evident what the optimal exponent could be as the nonlinear growth does take its part.\footnote{For preliminary results in this direction, it is worth mentioning the very recent paper~\cite{CRS16}, where optimal regularity results of the solution to the obstacle problem, and of the free boundary near regular points, have been achieved for linear integro-differential operators as in~\eqref{operatore} in the case when~$p=2$.}

 Secondly, it could be interesting to investigate the regularity in a generic point of the free boundary (known to be analytic in the case of the Laplacian, except on a well defined set of singular points, and smooth in the case of the fractional Laplacian). 

 Thirdly, a natural goal is to investigate the parabolic version of the nonlocal obstacle problem, as it is inspired in the so-called optimal stopping problem with deadline, by corresponding to the American option pricing problem with expiration at some given time.  
An extension in the setting presented here could be of relevant interest as it could describe a situation which also takes into account  the interactions coming from far together with a natural inhomogeneity. Accordingly with the optimal stopping problem model, a starting point in such an investigation could be the special case when the obstacle~$h$ coincides with the boundary value~$g$.

\medskip
The paper is organized as follows. In Section~\ref{sec_preliminaries} below, we fix the notation by also stating some general assumptions on the quantities we will deal with throughout the whole paper.
In Section~\ref{sec_obstacle}, we introduce the nonlinear fractional obstacle problem, and state and prove the existence and uniqueness of the related solutions.
The last two sections are devoted to the proofs of all the aforementioned boundedness and continuity results (Section~\ref{sec_regularity}), and up to the boundary (Section~\ref{sec_boundary}).

\medskip

\section{Preliminaries}\label{sec_preliminaries}

In this section, we state the general assumptions on the quantities we are dealing with. We keep these assumptions throughout the paper.
\smallskip

First of all, we recall that the class of integro-differential equations in which we are interested is the following
\begin{equation}\label{problema}
\mathcal{L}u(x)=\int_{\R^n} K(x,y)|u(x)-u(y)|^{p-2}\big(u(x)-u(y)\big)\dy = 0, \quad x\in \R^n.
\end{equation}
The nonlocal operator~$\mathcal{L}$ in the display above (being read a priori in the principal value sense) is driven by its {\it kernel} $K:\R^n\times \R^n \to [0,\infty)$, which is a measurable function  satisfying the following property:
\begin{equation}\label{hp_k}
\Lambda^{-1} \leq K(x,y)|x-y|^{n+sp} \leq \Lambda \quad \text{for a.~\!e. } x, y \in \R^n,
\end{equation}
for some $s\in (0,1)$, $p>1$, $\Lambda \geq1$. We immediately notice that in the special case when $p=2$ and $\Lambda=1$, we recover the well-known fractional Laplacian operator~$(-\Delta)^s$.
Also, notice that the assumption on $K$ can be weakened,
and in \eqref{problema} the dependence of $u(x)-u(y)$, in turn, can be weakened from $t \mapsto |t|^{p-2}t$ (see, for instance,~\cite{KMS15}).
However, for the sake of simplicity, we will take \eqref{problema} and we will
work under the assumption in~\eqref{hp_k}, since the weaker assumptions would bring no relevant differences in all the forthcoming proofs. 

\medskip

We now  recall the definition of {\it the nonlocal  tail \,{\rm{Tail}$(f; z, r)$} of a function $f$ in the ball of radius $r>0$ centered in $z\in \R^n$}. We have
\begin{equation} \label{def_tail} 
{\rm Tail}(f;z,r) := \bigg(r^{sp} \int_{\R^n \setminus B_r(z)} |f(x)|^{p-1} |x-z|^{-n-sp} \dx \bigg)^{\frac{1}{p-1}},
\end{equation}
for any function $f$ initially defined in $L^{p-1}_{\textrm{loc}}(\R^n)$. As mentioned in the introduction, this quantity will play an important role in the rest of the paper. The nonlocal tail has been introduced in~\cite{DKP15}, and used subsequently in several recent papers (see e.~\!g.,~\cite{BL15,DKP14,HRS15,KMS15,KMS15b,IMS15,IS14} and many others\footnote{
When needed, our definition of Tail can also be given in a more general way by replacing the ball~$B_r$ and the corresponding~$r^{sp}$ term by an open bounded set~$E\subset\R^n$ and its rescaled measure~$|E|^{sp/n}$, respectively. This is not the case in the present paper.
}), where it has been crucial to control in a quantifiable way the long-range interactions which naturally appear when dealing with nonlocal operators of the type considered here in~\eqref{problema}.
When having to control the positive and negative interactions separately, we denote the positive part and the negative part of a function $u$ by $u_+:=\max\{u,0\}$ and $u_-:=\max\{-u,0\}$, respectively. 
In the following, when the center point $z$ will be clear from the context, we shall use the shorter notation \, Tail$(f; r)\equiv$ Tail$(f; z, r)$. 
 Now, in clear accordance with the definition in~\eqref{def_tail}, for any $p>1$ and any $s\in (0,1)$, one can consider the corresponding  {\it tail space} $L^{p-1}_{sp}(\R^n)$ given by
\begin{equation*} 
L^{p-1}_{sp}(\R^n) := \Big\{ f \in L_{\rm loc}^{p-1}(\R^n) \; : \; {\rm Tail}(f;z,r)< \infty \quad \forall z \in \R^n, \forall r \in (0,\infty)\Big\}.
\end{equation*}
Notice that 
\begin{equation*}
L^{p-1}_{sp}(\R^n) = \Big\{ f \in L_{\rm loc}^{p-1}(\R^n) \; : \;   \int_{\R^n} |f(x)|^{p-1} (1+|x|)^{-n-sp} \dx < \infty \Big\}.
\end{equation*}
As expected, one can check that $W^{s,p}(\R^n) \subset L^{p-1}_{sp}(\R^n)$, where we denoted by $W^{s,p}(\R^n)$ the usual fractional Sobolev space of order $(s,p)$, 
defined by the norm
\begin{align}\label{def_seminorm}
\|v\|_{W^{s,p}(\R^n)} & :=  \|v\|_{L^p(\R^n)}
+ [v]_{{W}^{s,p}(\R^n)}
 \nonumber \\
&\,\, = \left(\int_{\R^n} |v|^p\dx \right)^{\frac1p} + \left(\int_{\R^n}\int_{\R^n}\frac{|v(x)-v(y)|^p}{|x-y|^{n+sp}}\dxy\right)^{\frac1p}, 
\end{align}
where $s\in (0,1)$ and $p\geq1$. The local fractional Sobolev space $W^{s,p}(\Omega)$ for $\Omega \subset \R^{n}$ is defined similarly.
By $W_0^{s,p}(\Omega)$ we denote the closure of $C_0^\infty(\Omega)$ in $W^{s,p}(\R^n)$. Conversely, if $v \in W^{s,p}(\Omega')$ with $\Omega \Subset \Omega'$ and $v=0$ outside of $\Omega$ almost everywhere, then $v$ has a representative in $W_0^{s,p}(\Omega)$ as well (see, for instance, \cite{DPV12}). 
\medskip

We  now recall the definitions of sub and supersolutions $u$ to the class of 
integro-differential problems we are interested in.
A function $u \in W^{s,p}_{\rm{loc}}(\Omega)\cap L^{p-1}_{sp}(\R^n)$ is a {\it  fractional weak $p$-supersolution} of~$\eqref{problema}$ if
\begin{align} \label{supersolution}
\langle \mathcal{L}u,\eta\rangle & \equiv  \int_{\R^n} \int_{\R^n}K(x,y)|u(x)-u(y)|^{p-2}\big(u(x)-u(y)\big)\big(\eta(x)-\eta(y)\big)\dxy \nonumber\\*
& \ge 0
\end{align} 
for every nonnegative $\eta \in C^\infty_0(\Omega)$. Here $\eta \in C^\infty_0(\Omega)$ can be replaced by $\eta \in W^{s,p}_0(\Omega')$ with
every $\Omega' \Subset \Omega$. 
It is worth noticing that the summability assumption of $u$ belonging to the tail space $L^{p-1}_{sp}(\R^n)$ is what one expects in the  nonlocal framework considered here (see~\cite{DKKP15}). 
\\ A function $u \in W^{s,p}_{\rm{loc}}(\Omega) \cap L^{p-1}_{sp}(\R^n)$ is a {\it  fractional weak 
$p$-subsolution} if $-u$ is a  fractional weak $p$-supersolution. Finally, a function $u$ is a {\it  fractional weak $p$-solution} if it is both  fractional weak $p$-sub and supersolution. In the following, we simply refer to those $u$ as (weak) supersolutions, subsolutions and solutions. 

Moreover, let us remark that we will assume that the kernel~$K$ is symmetric, and once again this is not restrictive,  in view of the weak formulation presented above, since one may always define the corresponding symmetric kernel $K_{\textrm{\tiny sym}}$ given by
$$
K_{\textrm{\tiny sym}}(x,y):=\frac1{2}\Big(K(x,y)+K(y,x)\Big).
$$

\medskip
We conclude this section by presenting some basic estimates which will be useful in the course of the forthcoming proofs. 
As customary when dealing with nonlinear operators, we will often have to treat in a different way the superquadratic case when $p>2$ and the subquadratic case $1<p<2$.
In order to simplify the notation in the weak formulation in~\eqref{supersolution}, from now on we denote by
\begin{equation}\label{def_l}
L(a,b):=|a-b|^{p-2}(a-b), \quad a,b \in \R.
\end{equation}
Notice that $L(a,b)$ is increasing with respect to $a$ and decreasing with respect to $b$.

\begin{lemma} \label{lemma:1<p<2}
Let $1<p \le 2$ and $a,b,a',b' \in \R$. Then
\begin{align}\label{lem_2.2}
|L(a,b)-L(a',b')| \le 4|a-a'-b+b'|^{p-1}.
\end{align}
\end{lemma}
\begin{proof}
Denoting by 
\[
f(t):=L\big(ta+(1-t)a', tb+(1-t)b'\big),
\]
we obtain by the chain rule
\begin{align} \label{Lab-La'b'}
|L(a,b)-L(a',b')| &= \Big| \int_0^1 f'(t) \dt \Big| \nonumber \\
& = \Big| \int_0^1 (a-a')\partial_a L + (b-b')\partial_b L \dt \Big| \nonumber \\
&= (p-1)|a-a'-b+b'|\int_0^1 |t(a-b)+(1-t)(a'-b')|^{p-2} \dt,
\end{align}
where we also used that
\[
\partial_a L(a,b)=(p-1)|a-b|^{p-2} \quad
\text{and} \quad \partial_b L(a,b)=-(p-1)|a-b|^{p-2}.
\]

Now, for $\alpha \in [-1,1]$, define 
\[
g(\alpha) := \int_0^{1} |t \alpha + 1-t|^{p-2}\dt.
\]
Note that $g(1) = 1$. If $\alpha<1$, then changing variables as $\tau=t\alpha+ 1-t$ yields
\[
g(\alpha) \, = \, \frac{1}{1-\alpha} \int_{\alpha}^{1}  |\tau|^{p-2}\,{\rm d}\tau
\,  \leq \,  2\int_{0}^{1}  \tau^{p-2}\,{\rm d}\tau 
\, = \,  \frac2{p-1}.
\]
By using the estimate above for $\alpha := \beta/\gamma$ with $|\beta|\leq |\gamma|$, we obtain
\begin{align}\label{210b}
|\gamma-\beta| \int_0^{1} |t \beta + (1-t)\gamma|^{p-2}\dt & = |\gamma-\beta| |\gamma|^{p-2} g(\alpha) \nonumber \\
& \leq \frac{4}{p-1} |\gamma-\beta|^{p-1}.
\end{align}
Finally, by combining \eqref{210b} with \eqref{Lab-La'b'}, letting $\beta=a-b$ and $\gamma=a'-b'$, we obtain the desired result.
\end{proof}

\begin{lemma} \label{lemma:p>2}
Let $p \ge 2$ and $a,b,a',b' \in \R$. Then 
\begin{align} \label{p>2aba'b}
|L(a,b)-L(a',b)| \le c\,|a-a'|^{p-1}+c\,|a-a'||a-b|^{p-2}
\end{align}
and
\begin{align} \label{p>2abab'}
|L(a,b)-L(a,b')| \le c\,|b-b'|^{p-1}+c\,|b-b'||a-b|^{p-2},
\end{align}
where $c$ depends only on $p$.
\end{lemma}
\begin{proof}
Denoting by $f(t):=L\big(ta'+(1-t)a,\,  b\big)$, we obtain by the chain rule
\begin{align*}
|L(a,b)-L(a',b)| &= \Big| \int_0^1 f'(t) \dt \Big| = \Big| \int_0^1 (a'-a)\partial_a L(ta'+(1-t)a,b) \dt \Big| \\[1ex]
&= (p-1)|a-a'| \int_0^1 |ta'+(1-t)a-b|^{p-2} \dt \\[1ex]
&\le c\,|a-a'|^{p-1}+c\,|a-a'||a-b|^{p-2},
\end{align*}
where we also used that
\[
\partial_a L(a,b)=(p-1)|a-b|^{p-2}.
\]
Thus, the inequality in~\eqref{p>2aba'b} does hold. Similarly, one can prove the inequality in~\eqref{p>2abab'}.
\end{proof}

Finally, we would like to make the following observation. In the rest of the paper, we often use the fact that there is a constant $c>0$ depending only on $p$ such that
\begin{align} \label{a-b bounds}
\frac1{c} \le \frac{\big(|a|^{p-2}a-|b|^{p-2}b\big)(a-b)}{(|a|+|b|)^{p-2}(a-b)^{2}} \le c,
\end{align}  
when $a,b \in \R$, $a \neq b$. In particular,
\begin{align} \label{a-b positive}
\big(|a|^{p-2}a-|b|^{p-2}b\big)(a-b) \geq 0, \quad a,b \in \R.
\end{align}

\medskip

\section{The obstacle problem}\label{sec_obstacle}
As mentioned in the introduction, by solving the fractional obstacle problem we will provide an important tool in the development of the fractional Nonlinear Potential Theory, and, in order to present such a topological approach, 
we start by introducing a necessary set of notation. 
Let $ \Omega \Subset \Omega'$ be open bounded subsets of $\R^n$. Let  
$h \colon \R^n \to [-\infty,\infty)$ be an extended real-valued function, which is considered to be the obstacle,
and let $g \in W^{s,p}(\Omega') \cap L^{p-1}_{sp}(\R^n)$ be the boundary values. We define
$$
\mathcal K_{g,h}(\Omega,\Omega') := \Big\{u \in W^{s,p}(\Omega') \,:\, u \geq h\, \text{ a.~\!e. in } \Omega, \, u = g\, \text{ a.~\!e. on } \R^n \setminus \Omega \Big\}.
$$
The interpretation for the case $h \equiv -\infty$ is that 
$$
\mathcal K_{g}(\Omega,\Omega')  \equiv \mathcal K_{g,-\infty}(\Omega,\Omega') := \Big\{u \in W^{s,p}(\Omega') \,:\, u = g\, \text{ a.~\!e. on } \R^n \setminus \Omega \Big\},
$$ 
i.~\!e., the class where we are seeking solutions to the Dirichlet boundary value problem. A few observations are in order. First, a natural assumption for any existence theory is that $\mathcal K_{g,h}(\Omega,\Omega')$ is a non-empty set. This is a property of functions $g$ and $h$. Second, we are assuming that $g$ has bounded fractional Sobolev norm in a set~$\Omega'$ which is strictly containing the set $\Omega$, and not necessarily in the whole~$\R^n$ as previously in the literature.

\vs
\subsection{Existence of solutions}
The obstacle problem can be reformulated as a standard problem in the theory of variational inequalities on Banach spaces, by seeking the energy minimizers in the set of suitable functions defined above. For this, by taking into account the nonlocality of the involved operators here,
it is convenient to define a functional $\mathcal A \colon \mathcal K_{g,h}(\Omega,\Omega') \to \left[W^{s,p}(\Omega')\right]'$ given by
\begin{equation}\label{def_a}
\mathcal Au(v)  := \mathcal A_1u(v) +  \mathcal A_2 u(v)
\end{equation}
for every $u \in \mathcal K_{g,h}(\Omega,\Omega')$ and $v \in W^{s,p}(\Omega')$, where
\[
 \mathcal A_1 u(v) := \int_{\Omega'}\int_{\Omega'} L(u(x),u(y))\big(v(x)-v(y)\big)K(x,y)\dxy
\]
and
\[
 \mathcal A_2 u(v) := 2 \int_{\R^n \setminus \Omega'}\int_{\Omega} L(u(x),g(y))v(x)K(x,y)\dxy.
\]
The motivation for the definitions above is as follows. Assuming that $v \in W_{0}^{s,p}(\Omega)$, and $u \in W^{s,p}(\Omega')$ is such that $u = g$ on $\R^n \setminus \Omega'$, we have that 
\begin{align} \label{eq:weak sol vs A} 
\notag 
& \int_{\R^n}\int_{\R^n} L(u(x),u(y))\big(v(x)-v(y)\big)K(x,y)\dxy
\\* \notag & \qquad = \int_{\Omega'}\int_{\Omega'} L(u(x),u(y))\big(v(x)-v(y)\big)K(x,y)\dxy
\\* \notag & \qquad \quad + 2\int_{\R^n \setminus \Omega'} \int_{\Omega} L(u(x),g(y)) v(x) K(x,y)\dxy.
\\*  & \qquad \equiv  \mathcal A_1 u(v) +  \mathcal A_2 u(v).
\end{align}
In the following we will use the usual brackets, as e.~\!g.~$\langle \mathcal{A}_1(u)-\mathcal{A}_1(w), v\rangle$ to denote $\mathcal{A}_1u(v) -\mathcal{A}_1w(v)$, and so on.

\begin{remark} 
The functional $\mathcal A u$ really belongs to the dual of $W^{s,p}(\Omega')$. Indeed, we have
\begin{align} \label{A1uv}
| \mathcal A_1 u(v) | &\le \int_{\Omega'}\int_{\Omega'} |u(x)-u(y)|^{p-1}|v(x)-v(y)|K(x,y)\dxy \nonumber\\[1ex]
&\le c \left(\int_{\Omega'}\int_{\Omega'} |u(x)-u(y)|^p \frac{{\rm d}x{\rm d}y}{|x-y|^{n+sp}}\right)^\frac{p-1}{p} \nonumber\\
&\qquad \times \left(\int_{\Omega'}\int_{\Omega'} |v(x)-v(y)|^p \frac{{\rm d}x{\rm d}y}{|x-y|^{n+sp}}\right)^\frac{1}{p} \nonumber\\[1ex]
&\le c\,\|u\|_{W^{s,p}(\Omega')}^{p-1} \|v\|_{W^{s,p}(\Omega')}
\end{align}
by H\"older's Inequality. Also,
\begin{align} \label{A2uv}
| \mathcal A_2 u(v)| &\le 2\int_{\R^n \setminus \Omega'}\int_{\Omega} |u(x)-g(y)|^{p-1}|v(x)|K(x,y)\dxy  \nonumber\\[1ex] \nonumber
&\le c \int_{\R^n \setminus \Omega'}\int_{\Omega} |u(x)|^{p-1}|v(x)||x-y|^{-n-sp}\dxy  \nonumber\\
& \quad + c \int_{\R^n \setminus \Omega'}\int_{\Omega} |g(y)|^{p-1}|v(x)||x-y|^{-n-sp}\dxy \nonumber \\[1ex]
&\le c\,r^{-sp} \left(\int_{\Omega} |u(x)|^p \dx\right)^\frac{p-1}{p} \left(\int_{\Omega} |v(x)|^p \dx\right)^\frac{1}{p} \nonumber \\
& \quad + c\,\bigg(\int_{\R^n \setminus \Omega'}|g(y)|^{p-1}|z-y|^{-n-sp}\dy\bigg)\int_{\Omega}|v(x)|\dx \nonumber \\[1ex]
&\le c\,r^{-sp}\Big(\|u\|_{W^{s,p}(\Omega')}^{p-1}+{\rm Tail}(g;z,r)^{p-1} \Big)\|v\|_{W^{s,p}(\Omega')}
\end{align}
holds, where $z \in \Omega$ and $r:=\dist(\Omega,\partial \Omega')>0$, and $c$ depends on $n,p,s,\Lambda,\Omega,\Omega'$. 
\end{remark}

\begin{remark} \label{remark:A2}
In the definition \eqref{def_a}, we could replace $A_2 u(v)$ by
\begin{equation} \label{A2 alternative}
2 \int_{\R^n \setminus \Omega'}\int_{\Omega''} L(u(x),g(y))v(x)K(x,y)\dxy
\end{equation}
for $\Omega''$ satisfying $\Omega \subset \Omega'' \Subset \Omega'$.
Anyway, we need a strictly positive distance between $\partial \Omega''$ and $\partial \Omega'$
to deduce $\mathcal A u \in [W^{s,p}(\Omega')]'$, as seen in the calculations for \eqref{A2uv} above.
\end{remark}

Now, we are ready  to provide the natural definition of solutions to the obstacle problem in the general nonlocal framework considered here. 
\begin{definition}
We say that $u \in \mathcal K_{g,h}(\Omega,\Omega')$ is {\it a solution to the obstacle problem} in $\mathcal K_{g,h}(\Omega,\Omega')$ if
\[
 \mathcal Au(v-u)  \geq 0
\]
whenever $v \in \mathcal K_{g,h}(\Omega,\Omega')$.
\end{definition}
Below, we state and prove the uniqueness of the solution to the obstacle problem and the fact that such a solution is a weak supersolution to~\eqref{problema}. Also, under natural assumptions on the obstacle $h$, one can prove that the solution to the obstacle problem is (fractional) harmonic away from the contact set, in clear accordance with the classical results when $s=1$.
 We have

\begin{theorem} \label{obst prob sol}
There exists a unique solution to the obstacle problem in $\mathcal K_{g,h}(\Omega,\Omega')$. Moreover, the solution to the obstacle problem is a weak supersolution to \eqref{problema} in $\Omega$.
\end{theorem}

\begin{corollary} \label{obst prob free}
Let $u$ be the solution to the obstacle problem in $\mathcal K_{g,h}(\Omega,\Omega')$. If $B_r \subset \Omega$ is such that
\[
\essinf_{B_r} (u - h) >0,
\]
then $u$ is a weak solution to \eqref{problema} in $B_r$. In particular, if $u$ is lower semicontinuous and $h$ is upper semicontinuous in $\Omega$, then $u$ is a weak solution to \eqref{problema}
in $\Omega_+:=\big\{x \in \Omega : u(x)>h(x)\big\}$.
\end{corollary}

\begin{remark}
When solving the obstacle problem in $\mathcal K_{g,-\infty}(\Omega,\Omega')$,
we obtain a unique weak solution to \eqref{problema} in $\Omega$ having the boundary values
$g \in W^{s,p}(\Omega') \cap L^{p-1}_{sp}(\R^{n})$ in $\R^{n} \setminus \Omega$.
This is a generalization of the existence results stated in \cite{DKP15}, where $g \in W^{s,p}(\R^{n})$, and -- as already mentioned in the introduction -- in general of all the analyses of fractional obstacle problems in the previous literature when $\Omega'$ does coincide with the whole~$\R^n$.
\end{remark}

Before going into the related proofs, we need to state and prove some properties of the operator $\mathcal{A}$ defined in \eqref{def_a}. We have the following
\begin{lemma} \label{Amcwc}
The operator $\mathcal A$ is monotone, coercive and weakly continuous on the set~$\mathcal K_{g,h}(\Omega,\Omega')$.
\end{lemma}
\begin{proof}
We start with the monotonicity, that is, we show that $\langle \mathcal A u - \mathcal A v, u-v \rangle \ge 0$ holds for every $u,v \in \mathcal K_{g,h}(\Omega,\Omega')$.
To this end, let $u,v \in \mathcal K_{g,h}(\Omega,\Omega')$. We have
\begin{align*}
& \langle \mathcal A_1 u - \mathcal A_1 v, u-v \rangle \\
&\qquad = \int_{\Omega'}\int_{\Omega'} \big(|u(x)-u(y)|^{p-2}(u(x)-u(y))-|v(x)-v(y)|^{p-2}(v(x)-v(y))\big) \\
&\qquad \qquad\qquad \times \big(u(x)-u(y)-v(x)+v(y)\big)K(x,y)\dxy
\end{align*}
and this quantity is nonnegative in view of~\eqref{a-b positive}.
Similarly, for $\mathcal A_2$,
\begin{align*}
&\langle \mathcal A_2 u - \mathcal A_2 v, u-v \rangle \\
&\qquad = 2\int_{\R^n \setminus \Omega'}\int_{\Omega} \Big(|u(x)-g(y)|^{p-2}\big(u(x)-g(y))-|v(x)-g(y)|^{p-2} \\
&\qquad \qquad\qquad \times (v(x)-g(y)\big)\Big)\big(u(x)-g(y)-v(x)+g(y)\big)K(x,y)\dxy \\
&\qquad \ge 0.
\end{align*}
Hence $\mathcal A$ is monotone.

\medskip
Next, we prove the weak continuity. Let $\{u_j\}$ be a sequence in $\mathcal K_{g,h}(\Omega,\Omega')$ converging to $u \in \mathcal K_{g,h}(\Omega,\Omega')$ in $W^{s,p}(\Omega')$. The weak continuity condition is that $\langle \mathcal A u_j - \mathcal Au,v \rangle \to  0$ for all $v \in W^{s,p}(\Omega')$. Thus, let $v \in W^{s,p}(\Omega')$.
Then for $\mathcal A_1$ and $1<p\leq 2$, applying \eqref{lem_2.2} and H\"older's Inequality, we obtain
\begin{align*}
& |\langle \mathcal A_1 u_j - \mathcal A_1 u, v \rangle| \\[1ex]
&\qquad \le \int_{\Omega'}\int_{\Omega'} \big|L(u_j(x),u_j(y))-L(u(x),u(y))\big||v(x)-v(y)|K(x,y)\dxy \\[1ex]
&\qquad \le c\int_{\Omega'}\int_{\Omega'} |u_j(x)-u_j(y)-u(x)+u(y)|^{p-1}|v(x)-v(y)|\frac{{\rm d}x{\rm d}y}{|x-y|^{n+sp}} \\[1ex]
&\qquad \le c\,\|u_j-u\|_{W^{s,p}(\Omega')}^{p-1} \|v\|_{W^{s,p}(\Omega')},
\end{align*}
which vanishes as $j \to \infty$.
On the other hand, when $p>2$, by 
using \eqref{a-b bounds},
we have, again by H\"older's Inequality, that
\begin{align*}
& |\langle \mathcal A_1 u_j - \mathcal A_1 u, v \rangle| \\*[1ex]
&\qquad \le \int_{\Omega'}\int_{\Omega'} \big|L(u_j(x),u_j(y))-L(u(x),u(y))\big||v(x)-v(y)|K(x,y)\dxy \\*[1ex]
&\qquad \le c\int_{\Omega'}\int_{\Omega'} \big(|u_j(x)-u_j(y)|+|u(x)-u(y)|\big)^{p-2} \\*
&\qquad \qquad\qquad \times |u_j(x)-u_j(y)-u(x)+u(y)||v(x)-v(y)|\frac{{\rm d}x{\rm d}y}{|x-y|^{n+sp}} \\[1ex]
&\qquad \le c\int_{\Omega'}\int_{\Omega'} \bigg(\frac{|u_j(x)-u_j(y)|^{p-2}}{|x-y|^{s(p-2)}}+\frac{|u(x)-u(y)|^{p-2}}{|x-y|^{s(p-2)}}\bigg) \\
&\qquad \qquad\qquad \times \frac{|u_j(x)-u_j(y)-u(x)+u(y)|}{|x-y|^{s}}\,\frac{|v(x)-v(y)|}{|x-y|^{s}}\frac{{\rm d}x{\rm d}y}{|x-y|^{n}} \\[1ex]
&\qquad \le c\left(\int_{\Omega'}\int_{\Omega'} \frac{|u_j(x)-u_j(y)|^{p}}{|x-y|^{n+sp}}
+ \frac{|u(x)-u(y)|^{p}}{|x-y|^{n+sp}}\dxy\right)^{\frac{p-2}{p}} \\
&\qquad \qquad \times \left(\int_{\Omega'}\int_{\Omega'} \frac{|u_j(x)-u(x)-u_j(y)+u(y)|^{p}}{|x-y|^{n+sp}}\dxy\right)^{\frac1{p}} \\
&\qquad \qquad  \times \left(\int_{\Omega'}\int_{\Omega'} \frac{|v(x)-v(y)|^{p}}{|x-y|^{n+sp}}\dxy\right)^{\frac1{p}} \\[1ex]
&\qquad \le c\,\big(\|u_j\|_{W^{s,p}(\Omega')}+\|u\|_{W^{s,p}(\Omega')}\big)^{p-2} \|u_j-u\|_{W^{s,p}(\Omega')} \|v\|_{W^{s,p}(\Omega')},
\end{align*}
which vanishes as $j \to \infty$.
Similarly, for $\mathcal A_2$ when $1<p \le 2$, by using again \eqref{lem_2.2}, we have
\begin{align*}
& |\langle \mathcal A_2 u_j - \mathcal A_2 u, v \rangle| \\
&\qquad \le 2\int_{\R^n \setminus \Omega'}\int_{\Omega} \big|L(u_j(x),g(y))-L(u(x),g(y))\big||v(x)|K(x,y)\dxy \\[1ex]
&\qquad \le c\int_{\R^n \setminus \Omega'}\int_{\Omega} |u_j(x)-u(x)|^{p-1}|v(x)||x-y|^{-n-sp}\dxy \\[1ex]
&\qquad \le c\,\|u_j-u\|_{W^{s,p}(\Omega')}^{p-1}\|v\|_{W^{s,p}(\Omega')},
\end{align*}
which tends to $0$ as $j \to \infty$. In the case when $p>2$, by using~\eqref{p>2aba'b}--\eqref{p>2abab'} we get
\begin{align*}
& |\langle \mathcal A_2 u_j - \mathcal A_2 u, v \rangle| \\
&\qquad \le c\int_{\R^n \setminus \Omega'}\int_{\Omega} |u_j(x)-u(x)|^{p-1}|v(x)||x-y|^{-n-sp}\dxy \\
&\qquad \quad + c\int_{\R^n \setminus \Omega'}\int_{\Omega} |u_j(x)-u(x)||u(x)-g(y)|^{p-2}|v(x)||x-y|^{-n-sp}\dxy \\[1ex]
&\qquad \le c\int_{\Omega} |u_j(x)-u(x)|^{p-1}|v(x)|\dx + c\int_{\Omega} |u_j(x)-u(x)||u(x)|^{p-2}|v(x)|\dx \\
&\qquad \quad + c\,\bigg(\int_{\R^n \setminus \Omega'}|g(y)|^{p-2}|z-y|^{-n-sp}\dy\bigg) \int_{\Omega} |u_j(x)-u(x)||v(x)|\dx \\[1ex]
&\qquad \le c\,\|u_j-u\|_{W^{s,p}(\Omega')}^{p-1}\|v\|_{W^{s,p}(\Omega')} \\
&\qquad \quad + c\,\|u_j-u\|_{W^{s,p}(\Omega')}\|u\|_{W^{s,p}(\Omega')}^{p-2}\|v\|_{W^{s,p}(\Omega')} \\
&\qquad \quad + c\,r^{-sp}\,{\rm Tail}(g;z,r)^{p-2}\|u_j-u\|_{W^{s,p}(\Omega')}\|v\|_{W^{s,p}(\Omega')},
\end{align*}
which again tends to $0$ as $j \to \infty$. 
Notice that in the display above the nonlocal integral has been estimated as follows
\begin{align*}
& \int_{\R^n \setminus \Omega'}|g(y)|^{p-2}|z-y|^{-n-sp}\dy \\
&\qquad \le \bigg(\int_{\R^n \setminus \Omega'}|g(y)|^{p-1}|z-y|^{-n-sp}\dy\bigg)^\frac{p-2}{p-1}
\bigg(\int_{\R^n \setminus \Omega'}|z-y|^{-n-sp}\dy\bigg)^\frac{1}{p-1} \\[1.5ex]
&\qquad \le c\,r^{-sp}\,{\rm Tail}(g;z,r)^{p-2},
\end{align*}
where $z \in \Omega$ and $r:=\dist(\Omega,\partial \Omega')>0$. Thus, $\langle \mathcal A u_j,v \rangle \to \langle \mathcal A u, v \rangle$ for every $v \in W^{s,p}(\Omega')$ as $j \to \infty$, i.~\!e., the weak continuity holds.

\medskip
Finally, we prove the coercivity, which means that there exists a function $v \in \mathcal K_{g,h}(\Omega,\Omega')$ such that
\[
\frac{\langle \mathcal A u - \mathcal A v, u-v \rangle}{\|u-v\|_{W^{s,p}(\Omega')}} \to \infty \quad \text{as} \quad \|u\|_{W^{s,p}(\Omega')} \to \infty.
\]
Since we are assuming that $\mathcal K_{g,h}(\Omega,\Omega')$ is non-empty, there is at least one $v \in \mathcal K_{g,h}(\Omega,\Omega')$. Let this be fixed. By \eqref{A1uv} and \eqref{A2uv} we see that
\begin{align} \label{AuAvv}
|\langle \mathcal A u - \mathcal A v, v \rangle| \le c\,\|u\|_{W^{s,p}(\Omega')}^{p-1}+c,
\end{align}
where the constant $c$ is independent of $u$. We now show that the contribution from $\langle \mathcal A u - \mathcal A v, u \rangle$ dominates when $\|u\|_{W^{s,p}(\Omega')}$ is large.
For $\mathcal A_1$, we obtain
\begin{align}\label{A1uA1vu}
&\langle \mathcal A_1 u - \mathcal A_1 v, u \rangle \nonumber\\
&\qquad= \int_{\Omega'}\int_{\Omega'} \big(L(u(x),u(y))-L(v(x),v(y))\big)\big(u(x)-u(y)\big)K(x,y)\dxy \nonumber\\[1ex]
&\qquad\ge \frac1{c} \int_{\Omega'}\int_{\Omega'} \frac{|u(x)-u(y)|^{p}}{|x-y|^{n+sp}} \dxy \nonumber\\
&\qquad\quad -c \int_{\Omega'}\int_{\Omega'} |v(x)-v(y)|^{p-1}|u(x)-u(y)|\frac{{\rm d}x{\rm d}y}{|x-y|^{n+sp}} \nonumber\\[1ex]
&\qquad\ge  
\frac 1c \left[ u-g \right]_{W^{s,p}(\Omega')}^{p} - c \left[ g \right]_{W^{s,p}(\Omega')}^{p} 
 \nonumber\\
&\qquad\quad -c \int_{\Omega'}\int_{\Omega'} |v(x)-v(y)|^{p-1}|u(x)-u(y)|\frac{{\rm d}x{\rm d}y}{|x-y|^{n+sp}} \nonumber\\[1ex]
&\qquad\ge \frac1{c} \|u-g\|_{W^{s,p}(\Omega')}^{p}-c\,\|g\|_{W^{s,p}(\Omega')}^{p} - c\,\|v\|_{W^{s,p}(\Omega')}^{p-1}\|u\|_{W^{s,p}(\Omega')} \nonumber \\
&\qquad\ge \frac1{c} \|u\|_{W^{s,p}(\Omega')}^{p}-c\,\|g\|_{W^{s,p}(\Omega')}^{p} - c\,\|v\|_{W^{s,p}(\Omega')}^{p-1}\|u\|_{W^{s,p}(\Omega')},
\end{align}
by using in particular H\"older's Inequality and the fractional Sobolev embeddings (see for instance~\cite[Section~6]{DPV12}, and also \cite[Appendix 6.3]{PSV13} for a simple proof). 
For~$\mathcal A_2$, in turn, we obtain
\begin{align}\label{A2uA2vu}
&\langle \mathcal A_2 u - \mathcal A_2 v, u \rangle \nonumber\\
&\qquad = 2\int_{\R^n \setminus \Omega'}\int_{\Omega} \big(L(u(x),g(y))-L(v(x),g(y))\big)\big(u(x)-v(x)\big){K}(x,y)\dxy \nonumber \\ 
&\qquad \quad +\, 2\int_{\R^n \setminus \Omega'}\int_{\Omega} \big(L(u(x),g(y))-L(v(x),g(y))\big)v(x){K}(x,y)\dxy\nonumber \\[1ex]
&\qquad \geq   -2\int_{\R^n \setminus \Omega'}\int_{\Omega} \big|L(u(x),g(y))-L(v(x),g(y))\big||v(x)|{K}(x,y)\dxy \nonumber\\[1ex]
&\qquad \geq  -c \int_{\R^n \setminus \Omega'}\int_{\Omega} \Big(|u(x)|^{p-1}|v(x)|+|g(y)|^{p-1}|v(x)|+|v(x)|^{p}\Big)\frac{{\rm d}x{\rm d}y}{|x-y|^{n+sp}} \nonumber\\[1ex]
&\qquad \geq  -c\,\|u\|_{L^p(\Omega')}^{p-1}\|v\|_{L^p(\Omega')} - c\,r^{-sp}{\rm Tail}(g;z,r)^{p-1}\|v\|_{L^p(\Omega')} - \|v\|_{L^p(\Omega')}^p 
\end{align}
with $z\in\Omega$ and $r:=\dist(\Omega,\partial\Omega')$, where we also used that the term on the second line is nonnegative by the monotonicity.
By combining the estimates \eqref{AuAvv}, \eqref{A1uA1vu} with \eqref{A2uA2vu}, it yields
\begin{align*}
\langle \mathcal A u - \mathcal A v, u-v \rangle \ge \frac1{c} \|u\|_{W^{s,p}(\Omega')}^{p} - c\,\|u\|_{W^{s,p}(\Omega')}^{p-1}-c\,\|u\|_{W^{s,p}(\Omega')}-c,
\end{align*}
for a constant $c$ independent of $u$, and thus
\begin{align*}
\frac{\langle \mathcal A u - \mathcal A v, u-v \rangle}{\|u-v\|_{W^{s,p}(\Omega')}}
\to \infty \quad \text{as} \quad \|u\|_{W^{s,p}(\Omega')} \to \infty.
\end{align*}
This finishes the proof.
\end{proof}

Now, we are ready to prove the existence of a unique solution to the obstacle problem.

\begin{proof}[\bf Proof of Theorem~\ref{obst prob sol}]
We first notice that $\mathcal K_{g,h}(\Omega,\Omega') \subset W^{s,p}(\Omega')$ is nonempty, closed and convex. Also, in view of Lemma~\ref{Amcwc} the operator $\mathcal A$ is monotone, coercive and weakly continuous on $\mathcal K_{g,h}(\Omega,\Omega')$. This will permit us to apply the standard theory of monotone operators (see, for instance, Corollary III.1.8 in~\cite{KS80}, or~\cite{HKM06}) in order to deduce the existence of a function~$u \in \mathcal K_{g,h}(\Omega,\Omega')$ such that
\[
\langle \mathcal Au,v-u \rangle \ge 0,
\]
whenever $v \in \mathcal K_{g,h}(\Omega,\Omega')$. In order to show  the uniqueness, suppose that there are two functions $u_1$ and $u_2$ solving the obstacle problem. As a consequence, 
\[
\langle \mathcal Au_1,u_2-u_1 \rangle \ge 0 \quad \text{and} \quad \langle \mathcal Au_2,u_1-u_2 \rangle \ge 0,
\]
and then, by summing the preceding inequalities, we obtain
\[
\langle \mathcal Au_1-\mathcal Au_2,u_1-u_2 \rangle \le 0.
\]
This is possible only if $u_1=u_2$ almost everywhere. Thus, the solution $u$ is unique.

Now we show that the function $u$ is a weak supersolution to \eqref{problema} in $\Omega$.
First, clearly $u \in W^{s,p}_{\rm loc}(\Omega) \cap L^{p-1}_{sp}(\R^{n})$.
Then, notice that for any given nonnegative function $\phi \in C_0^{\infty}(\Omega )$, the function $v:=u+\phi$ belongs to $\mathcal K_{g,h}(\Omega,\Omega')$. Consequently, we have
as in~\eqref{eq:weak sol vs A}  that 
\begin{align*}
0 \le \langle \mathcal Au, \phi \rangle  
&= \int_{\R^{n}} \int_{\R^{n}}L(u(x),u(y))\big(\phi(x)-\phi(y)\big)K(x,y)\dxy.
\end{align*}
Thus, $u$ is a weak supersolution in $\Omega$.
\end{proof}

\begin{proof}[\rm \bf Proof of Corollary \ref{obst prob free}]
By Theorem \ref{obst prob sol} $u$ is a weak supersolution  
in $B_r \subset \Omega$. To show that $u$ is also a weak subsolution in $B_r$, let $\eta \in C_0^{\infty}(B_r)$ be a nonnegative test function that is not identically $0$. Set 
$\eps := \|\eta\|_\infty^{-1} \essinf_{B_r}(u-h)>0$. Then $v = u - \eps \eta \in \mathcal K_{g,h}(\Omega,\Omega')$ and $\langle \mathcal Au,v-u \rangle \geq 0$ yields $\langle \mathcal Au,\eta \rangle \leq 0$. Therefore, by~\eqref{eq:weak sol vs A} we obtain that $u$ is a weak subsolution in $B_r$, and thus a weak solution there. 
\end{proof}

The solution to the obstacle problem is the smallest supersolution above the obstacle in the following sense.
\begin{proposition} \label{smallest super}
Let $\Omega \Subset \Omega'' \subset \Omega'$. Let $u$ be the solution to the obstacle problem in $\mathcal K_{g,h}(\Omega,\Omega')$ and let $v$ be a weak supersolution in $\Omega''$ such that $\min\{u,v\} \in \mathcal K_{g,h}(\Omega,\Omega')$. Then $u \leq v$ almost everywhere.
\end{proposition}
\begin{proof}
Since $u$ is the solution to the obstacle problem and $\min\{u,v\}=u$ in $\R^{n}\setminus\Omega$,
\begin{align} \label{uminuv}
0 &\leq \langle \mathcal Au, \min\{u,v\}-u \rangle \\*
&= \int_{\R^{n}}\int_{\R^{n}} L(u(x),u(y))\Big(\min\{u,v\}(x)-u(x)-\min\{u,v\}(y)+u(y)\Big) \nonumber \\*
& \qquad \qquad \quad \times K(x,y)\dxy \nonumber.
\end{align}
Since $v$ is a weak supersolution in $\Omega''$ and $u-\min\{u,v\} \in W^{s,p}_0(\Omega)$ is nonnegative, we have
\begin{align} \label{vminuv}
0 &\leq \int_{\R^{n}}\int_{\R^{n}} L(v(x),v(y))\Big(u(x)-\min\{u,v\}(x)-u(y)+\min\{u,v\}(y)\Big) \nonumber \\
& \qquad \qquad \quad \times K(x,y)\dxy.
\end{align}
Summing the preceding inequalities \eqref{uminuv} and \eqref{vminuv}, we obtain
\begin{align*}
0 &\leq \int_{\R^{n}}\int_{\R^{n}} \big(L(v(x),v(y))-L(u(x),u(y))\big) \\*
&\qquad\qquad\qquad \times \big(u(x)-\min\{u,v\}(x)-u(y)+\min\{u,v\}(y)\big)K(x,y)\dxy \\[1ex]
&= \int_{\{u > v\}} \int_{\{u>v\}}\big(L(v(x),v(y))-L(u(x),u(y))\big) \\
&\qquad\qquad\qquad \times \big(u(x)-v(x)-u(y)+v(y)\big)K(x,y)\dxy \\
&\quad + 2\int_{\{u \leq v\}} \int_{\{u>v\}}\big(L(v(x),v(y))-L(u(x),u(y))\big)
\big(u(x)-v(x)\big)K(x,y)\dxy \\*[1ex]
& \leq 0
\end{align*}
since the first term is nonpositive by \eqref{a-b positive},
whereas in the second term, $L(v(x),v(y))<L(u(x),u(y))$ and $u(x)>v(x)$.
Consequently, $|\{u>v\}|=0$.
\end{proof}

\medskip

\section{Interior regularity}\label{sec_regularity}
In this section, we state and prove that the regularity of the solution to the obstacle problem inherits the regularity of the obstacle, both in the case of boundedness and (H\"older) continuity. This is in clear accordance with the by-now classical results for the obstacle problems in the pure fractional Laplacian $(-\Delta)^s$ case, as seen e.~\!g. in~\cite{ACS08,CSS08,Sil07}, via the Dirichlet-to-Neumann extension.
However, our approach here is different and, though we are dealing with a wider class of nonlinear integro-differential operators with coefficients, the proofs are  new and even simpler, since we can make effort of the recent quantitative estimates presented in the previous sections and in~\cite{DKP14,DKP15}, by taking care of the nonlocal tail quantities.

\begin{theorem} \label{thm:obs bnd}
Let $u$  be the solution to the obstacle problem in $\mathcal{K}_{g,h}(\Omega,\Omega')$. Assume that $B_r(x_0) \subset \Omega'$ and set 
\[
M :=\max\bigg\{ \esssup_{B_r(x_0) \cap \Omega}h ,  \esssup_{B_r(x_0) \setminus \Omega} g \bigg\}.
\]
Here the interpretation is that $ \esssup_{A} f = -\infty$ if $A = \emptyset$. 
If $M$ is finite, then $u$ is essentially bounded from above in $B_{r/2}(x_0)$  and 
\begin{equation} \label{eq_obs_bnd}
\esssup_{B_{r/2}(x_0)}(u-m)_+ \leq  \delta\, {\rm Tail}((u-m)_+;{x_0},r/2)+c\, \delta^{-\gamma} \left(\mean{B_r(x_0)} (u-m)_+^t\,{\rm d}x\right)^{\frac 1t}
\end{equation}
holds for all $m\geq M$, $t \in (0,p)$ and $\delta \in (0,1]$ with constants $\gamma \equiv \gamma(n,p,s,t)$ and $c \equiv c(n,p,s,t,\Lambda)$. 
\end{theorem}
\begin{proof}
Suppose that $M<\infty$. Letting $k\geq 0$, $m\geq M$,  and $\phi \in C_0^\infty(B_r(x_0))$, $0 \leq \phi \leq 1$, we see that $v = u-m - (u-m-k)_+ \phi^p$ belongs to $\mathcal{K}_{g-m,h-m}(\Omega,\Omega')$ and that $u_m := u-m$ solves the corresponding obstacle problem. Thus we have that 
\begin{align*}
&\int_{\R^n} \int_{\R^n} L(u_m(x),u_m(y))  \big( (u_m(x) {-}k)_+ \phi^p(x) {-} (u_m(y){-}k)_+ \phi^p(y)\big)\nonumber \\
& \qquad \quad \, \ \times K(x,y) \dxy\  \leq\ 0.
\end{align*}
As observed in the proof of~\cite[Theorem 1.4]{DKP15}, this is enough to prove first a Caccioppoli-type estimate with tail, and subsequently a local boundedness result (see~\cite[Theorem 1.1]{DKP15}) which yields
\begin{equation} \label{sup obs 1}
\esssup_{B_{\rho/2}(y)}(u_m)_+ \leq \tilde \delta\, {\rm Tail}((u_m)_+;{y},\rho/2)+c\, \tilde \delta^{-\tilde\gamma} \left(\mean{B_\rho(y)} (u_m)_+^p\dx\right)^{\frac 1p},
\end{equation}
whenever $B_\rho(y) \subset B_r(x_0)$, 
for any  $\tilde \delta \in (0,1]$, and with positive $\tilde\gamma$ depending only on $n,p,s$ and $c$ only on $n,p,s,\Lambda$. Now, a covering argument, which goes back to the one  in the proof of~\cite[Theorem 1.1]{DKP14}, will allow us to prove that~\eqref{sup obs 1} actually implies~\eqref{eq_obs_bnd}. For this, set $\rho=(\sigma-\sigma')r$ with $1/2\leq \sigma'<\sigma\leq 1$, and take $y \in \sigma'B\equiv B_{\sigma' r}(x_0)$. We can estimate the nonlocal contribution in~\eqref{sup obs 1} as follows
\begin{align}\label{tailm}
& {\rm Tail}( (u_m)_+; y, \rho/2)^{p-1} \nonumber \\[1ex]
& \qquad  \leq \left(\frac{\rho}{2}\right)^{sp} \sup_{\sigma B} (u_m)_+^{p-1} \int_{\sigma B\setminus B_{\rho/2}(y)}|x-y|^{-n-sp}\dx  
\nonumber\\
&  \qquad  \quad  + \left(
\frac{\rho}{2}\right)^{sp}
 \sup_{x\in \R^\n\setminus\sigma B}\left(\frac{|x-x_0|}{|x-y|}\right)^{n+sp}
\int_{\R^n\setminus \sigma B} (u_m)_+^{p-1}|x-x_0|^{-n-sp}\dx \nonumber \\[1ex]
& \qquad   \leq \ c\,\sup_{\sigma B} (u_m)_+^{p-1}
+ c\,(\sigma-\sigma')^{-{n}}{\rm Tail}((u_m)_+; x_0, r/2)^{p-1}.
\end{align}
For what concerns the local contribution in~\eqref{sup obs 1}, we can  apply Young's Inequality (with exponents $p/t$ and $p/(p-t)$) to get
\begin{align*}
\tilde \delta^{-\tilde\gamma}\left( \mean{B_{\rho}(y)} (u_m)_+^{p}\dx \right)^{\frac{1}{p}}
&\leq \tilde \delta^{-\tilde\gamma} \sup_{B_{\rho}(y)}(u_m)_+^{\frac{p-t}{p}} \left( \mean{B_{\rho}(y)} (u_m)_+^{t}\dx \right)^{\frac{1}{p}} \\[1ex]
&\leq \frac{1}{4}\sup_{\sigma B} (u_m)_+ 
+ c\,\tilde \delta^{-\frac{\tilde\gamma p}{t}}\left( \mean{B_{\rho}(y)} (u_m)_+^t\dx \right)^{\frac{1}{t}} \\[1ex]
&\leq \frac{1}{4}\sup_{\sigma B} (u_m)_+
+ c\,\tilde \delta^{-\frac{\tilde\gamma p}{t}}
(\sigma-\sigma')^{-\frac{n}{t}}
\left( \mean{B_r} (u_m)_+^t\dx \right)^{\frac{1}{t}}.
\end{align*}
Thus, by reabsorbing with $\tilde \delta\leq 1/4c$ we deduce by three last displays that
\begin{align*}
\sup_{\sigma' B}(u_m)_+ &\leq
\frac{1}{2}\sup_{\sigma B} (u_m)_+
+ c\,\tilde \delta^{-\frac{\tilde\gamma p}{t}} (\sigma-\sigma')^{-\frac{n}{t}}\left( \mean{B_r} (u_m)_+^t\dx \right)^{\frac{1}{t}} \\
&\quad +\, c\,\tilde \delta (\sigma-\sigma')^{-\frac{n}{p-1}}{\rm Tail}((u_m)_+; x_0, r/2),
\end{align*}
so that finally a standard iteration argument, see for instance~\cite[Lemma 3.38]{HKM06}, and choosing $\tilde \delta = \delta/c$ will give the desired result~\eqref{eq_obs_bnd}.
\end{proof}

The solution to the obstacle problem inherits the continuity of the obstacle. 

\begin{theorem} \label{thm:obs H cont}
Suppose that $h$ is locally H\"older continuous in $\Omega$. Then the solution $u$ to the obstacle problem in $\mathcal{K}_{g,h}(\Omega,\Omega')$ is locally H\"older continuous in $\Omega$ as well.
Moreover, for every $x_0 \in \Omega$ there is $r_0>0$ such that
\begin{align} \label{eq:obs cont} 
\osc_{B_\rho(x_0)} u & \leq c \left(\frac{\rho}{r}\right)^{\alpha} \left[ {\rm Tail}(u -h(x_0);{x_0},r) + \bigg(\mean{B_{r}(x_0)} |u-h(x_0)|^p \dx \bigg)^{\frac 1p}\right]  
\\ & \quad \nonumber + c \int_{\rho}^r \left(\frac{\rho}{t}\right)^{\alpha} \omega_h\left(  \frac{t}{\sigma} \right) \, \frac{dt}{t}
\end{align}
for every $r\in (0,r_0)$ and $\rho \in (0,r/4]$, where $\omega_h(\rho) \equiv \omega_h(\rho,x_0) := \osc_{B_{\rho}(x_0)} h$,
and $\alpha$, $c$ and $\sigma$ depend only on $n$, $p$, $s$, and $\Lambda$.
\end{theorem}
\begin{proof}
Let us first analyze the contact set, by which we mean that $x_0$ belongs to the contact set if and only if for every $r \in (0,R)$, $R:= \dist(x_0,\partial \Omega)$, we have
\[
\inf_{B_{r}(x_0)} (u-h) = 0.
\]
Our first goal is to show that for any such point $x_0$ and for any
$r \in (0,R)$ we find $\sigma \in (0,1)$ and $c$, both depending only on $n,p,s,\Lambda$, such that  
\begin{align} \label{eq:osc decay 000}
& \nonumber \osc_{B_{\sigma r}(x_0)} u + {\rm Tail}(u - h(x_0);{x_0},\sigma r) 
\\ & \qquad \leq \frac12  \left(\osc_{B_{r}(x_0)} u + {\rm Tail}(u - h(x_0);{x_0},r) \right)  + c\,\omega_h(r).
\end{align}
To this end, observe first that $u \geq d: = h(x_0) - \omega_h(r)$ almost everywhere in $B_r(x_0)$. Set $u_d := u-d$, which is then a nonnegative weak supersolution in $B_r(x_0)$ by Theorem~\ref{obst prob sol}. Now apply Theorem~\ref{thm:obs bnd} and the weak Harnack estimate in \cite[Theorem 1.2]{DKP14}. We obtain by \eqref{eq_obs_bnd} (applied with $m= d + 2\omega_h(r) \geq \sup_{B_{2\rho}(x_0)}h$) that 
\begin{equation} \label{eq:sup u_k 0}
\sup_{B_\rho(x_0)} u_d  \leq  2\omega_h(r) + \delta\,{\rm Tail}((u_d)_+ ;{x_0},\rho)+c\, \delta^{-\gamma} \left(\mean{B_{2\rho}(x_0)} u_d^t\dx \right)^{\frac 1t}
\end{equation}
for $\rho \in(0,r]$, $t\in(0,p)$ and $\delta\in(0,1]$, and the weak Harnack gives
\[
\left(\mean{B_{2\rho}(x_0)} u_d^t\dx\right)^{\frac 1t} \leq c \inf_{B_{4\rho}(x_0)} u_d + c \left(\frac{\rho}{r}\right)^{\frac{sp}{p-1}}{\rm Tail}( (u_d)_-;{x_0},r)
\]
whenever $\rho \in(0,r/4]$.
Since $\inf_{B_{\rho}(x_0)}u_d \leq \omega_h(r)$ due to $\essinf_{B_{\rho}(x_0)} (u-h) = 0$, we obtain from the previous display that 
\[
\left(\mean{B_{2\rho}(x_0)} u_d^t\dx\right)^{\frac 1t} \leq c\,\omega_h(r) +c \left(\frac{\rho}{r}\right)^{\frac{sp}{p-1}} {\rm Tail}(u_d;{x_0},r).
\]
Thus, recalling that $u_d \geq 0$ in $B_r(x_0)$, we arrive at
\[
\osc_{B_\rho(x_0)} u \leq c\,\delta^{-\gamma} \omega_h(r) + c\,\delta\,{\rm Tail}(u_d;{x_0},\rho)
+  c\,\delta^{-\gamma} \left(\frac{\rho}{r}\right)^{\frac{sp}{p-1}} {\rm Tail}(u_d;{x_0},r).
\]
Now we observe that
\begin{equation} \label{eq:tail u_k}
{\rm Tail}(u_d;{x_0},\rho) \leq c \sup_{B_{r}(x_0)} |u_d|  + c \left(\frac{\rho}{r}\right)^{\frac{sp}{p-1}} {\rm Tail}(u_d;{x_0},r),
\end{equation}
where we can estimate 
\begin{equation} \label{eq:sup u_k}
\sup_{B_r(x_0)} |u_d|=\sup_{B_r(x_0)} |u-h(x_0)+\omega_h(r)| \leq \osc_{B_{r}(x_0)} u+2 \omega_h(r).
\end{equation}
Now, for any $\eps \in (0,1)$, we can first choose $\delta$  small and then $\widetilde \sigma \in (0,1)$, correspondingly, so that we have
\[
c\, \delta \leq \frac{\eps}{2} \qquad \mbox{and} \qquad c\,\delta^{-\gamma} \widetilde \sigma^{\frac{sp}{p-1}} \leq \frac{\eps}{2},
\]
and therefore, for $\rho = \widetilde \sigma r$,
\begin{equation} \label{eq:osc tildesigmarho}
\osc_{B_{\widetilde \sigma r}(x_0)} u \leq \eps \left(\osc_{B_{r}(x_0)} u + {\rm Tail}(u - h(x_0);{x_0},r) \right)  + c_\eps\, \omega_h(r)
\end{equation}
holds. Using this together with \eqref{eq:tail u_k}, we further have that for any $\sigma \in (0,\widetilde \sigma)$
\begin{align*} 
{\rm Tail}(u - h(x_0);{x_0},\sigma r) & \leq c  \osc_{B_{\tilde \sigma r}(x_0)} u +c \left(\frac{\sigma}{\widetilde \sigma}\right)^{\frac{sp}{p-1}} {\rm Tail}(u - h(x_0);{x_0},\widetilde \sigma r)
\\* & \leq c\, \eps \left(\osc_{B_{r}(x_0)} u + {\rm Tail}(u - h(x_0);{x_0},r) \right)  + c\, c_\eps\, \omega_h(r) \\*
&\quad + c \left(\frac{\sigma}{\widetilde \sigma}\right)^{\frac{sp}{p-1}} \left(\osc_{B_r(x_0)}u + {\rm Tail}(u - h(x_0);{x_0},r)\right).
\end{align*}
Therefore, adding \eqref{eq:osc tildesigmarho} and taking $\sigma$ and $\eps$ so small that
\[
c \left(\frac{\sigma}{\widetilde \sigma}\right)^{\frac{sp}{p-1}} \leq \eps \qquad \text{and} \qquad (c+2)\,\eps \leq \frac12,
\]
yields \eqref{eq:osc decay 000}.

Next, iterating \eqref{eq:osc decay 000} we obtain
\begin{align} \label{eq:osc decay 001}
& \nonumber \osc_{B_{\sigma^k r}(x_0)} u + {\rm Tail}(u - h(x_0);{x_0},\sigma^k r) 
\\ & \qquad  \leq 2^{1-k} \left(\osc_{B_{\sigma r}(x_0)} u + {\rm Tail}(u - h(x_0);{x_0},\sigma r) \right)  + c \sum_{j=0}^{k-2} 2^{-j} \omega_h(\sigma^{k-j-1} r)
\end{align}
for any $k \in \mathbb N$.
Using finally the fact $\osc_{B_r} u = \osc_{B_r} u_d \leq \sup_{B_r} u_d$ and the supremum estimate \eqref{eq:sup u_k 0},  
we conclude the contact set analysis with
\begin{align} \label{eq:osc decay 002}
& \nonumber \osc_{B_{\sigma^k r}(x_0)} u + {\rm Tail}(u - h(x_0);{x_0},\sigma^k r) 
\\ & \qquad  \leq c\, 2^{1-k} \left( {\rm Tail}(u -h(x_0);{x_0},r) + \bigg(\mean{B_{r}(x_0)} |u-h(x_0)|^t\,{\rm d}x\bigg)^{\frac 1t}\right) \\
& \qquad\quad + c \sum_{j=0}^{k-1} 2^{-j} \omega_h(\sigma^{k-j-1} r). \nonumber
\end{align}
Notice here that if $h$ is continuous and uniformly bounded in $B_r$, then
\begin{equation*} \label{eq:omega_h sum}
\lim_{k\to \infty} \sum_{j=0}^{k-1} 2^{-j} \omega_h(\sigma^{k-j-1} r) = 0,
\end{equation*}
implying that $\lim_{r \to 0} \osc_{B_{r}(x_0)} u = 0$ in this case. 

We then analyze the continuity properties outside of the contact set. In this case we find $r_0 \in (0,R)$ such that 
$$
\inf_{B_{r_0}(x_0)} (u-h) > 0.
$$
Then Corollary \ref{obst prob free} says that $u$ is a weak solution in $B_{r_0}(x_0)$, and consequently we can use the results in~\cite{DKP15}, by also noticing that in the proofs there it makes no difference to assume
$u \in W^{s,p}_{\rm loc}(\Omega) \cap L^{p-1}_{sp}(\R^{n})$ instead of $u \in W^{s,p}(\R^{n})$.
In particular, \cite[Theorem 1.2]{DKP15} implies that
\begin{equation*}  \label{eq:osc decay 003}
\osc_{B_{\rho}(x_0)} u \leq c \left(\frac{\rho}{r}\right)^{\alpha} \left({\rm Tail}(u-h(x_0);{x_0},r) +  \bigg(\mean{B_{r}(x_0)} | u -h(x_0)|^p\dx\bigg)^{\frac 1p} \right).
\end{equation*}
for every $r \in (0,r_0)$ and $\rho \in (0,r/2]$.
The claim  
follows from this and \eqref{eq:osc decay 002} (with $\alpha \leq -\log 2 / \log \sigma$) after straightforward manipulations. 
\end{proof}

Slightly modifying the proof above, we easily obtain the following.
\begin{theorem} \label{thm:obs cont}
Suppose that $h$ is continuous in $\Omega$. Then the solution to the obstacle problem in $\mathcal{K}_{g,h}(\Omega,\Omega')$ is continuous in $\Omega$ as well.
\end{theorem}
\begin{proof}
This plainly follows from the previous theorem, since if $\omega_h(t) \to 0$ as $t\to 0$ and $\omega_h$ is locally uniformly bounded, then it is easy to check that 
$$
\int_{\rho}^r \left(\frac{\rho}{t}\right)^{\alpha} \omega_h\left(  \frac{t}{\sigma} \right) \, \frac{dt}{t} \to 0 
$$
as $\rho \to 0$ for small enough $r$. 
\end{proof}

\medskip

\section{Boundary regularity}\label{sec_boundary}

We continue our investigation by considering the regularity of the solution to the obstacle problem on the boundary of $\Omega$. 
In what follows, we assume $x_0 \in \partial \Omega$.
Firstly, we would need a Caccioppoli-type estimate with tail, whose proof is
a verbatim repetition of the proof of~\cite[Theorem~1.4]{DKP15} after noticing that $v = u \mp w_\pm \phi^p$, $\phi \in C_0^\infty(B_r(x_0))$, $0\leq \phi \leq 1$, belongs to $\mathcal{K}_{g,h}(\Omega,\Omega')$  for all indicated $k_+$ and $k_-$. For other fractional Caccioppoli-type inequalities, though not taking into account the tail contribution, see~\cite{Min07,Min11,FP14}. We have the following
\begin{lemma} \label{lemma:cacc bnd}
Suppose that $u \in \mathcal{K}_{g,h}(\Omega,\Omega')$ solves the obstacle problem in $\mathcal{K}_{g,h}(\Omega,\Omega')$.
Let $x_0 \in \partial\Omega$, let $r \in (0,r_0)$ with $r_0 := \dist(x_0,\partial \Omega')$, and suppose that 
\[
k_+ \geq \max\bigg\{ \esssup_{B_r(x_0)} g, \esssup_{B_r(x_0) \cap \Omega} h  \bigg\} \quad \text{and} \quad
k_- \leq \essinf_{B_r(x_0)} g.
\]
Then, for $w_+ := (u - k_+)_+$ and $w_- := (k_--u)_+$, we have 
\begin{eqnarray}\label{cacio1}
\nonumber && \int_{B_r(x_0)}\int_{B_r(x_0)} |w_{\pm}(x)\phi(x)-w_{\pm}(y)\phi(y)|^p K(x,y) \dxy \\[1ex]
 &&\qquad   \leq c\int_{B_r(x_0)}\int_{B_r(x_0)} w_{\pm}^p(x) |\phi(x)-\phi(y)|^p K(x,y) \dxy \\
&&\qquad  \quad+\,c \int_{B_r(x_0)}w_{\pm}(x)\phi^p(x)\dx \left(\sup_{y\,\in\,{\rm supp}\,\phi}\int_{\mathds{R}^n\setminus B_r(x_0)} w_{\pm}^{p-1}(x)K(x,y)\dx \right), 
\nonumber
\end{eqnarray}
whenever $\phi \in C_0^\infty(B_r(x_0))$ and $0\leq \phi \leq 1$.
\end{lemma}
\begin{remark} \label{remark:cacc contact}
If the maximum $\max\{ \esssup_{B_r(x_0)} g, \esssup_{B_r(x_0) \cap \Omega} h \}$ is infinite,
or $\essinf_{B_r(x_0)} g = -\infty$,  
then the interpretation is that there is no test function of the type $w_+$ or $w_-$, respectively. 
\end{remark}

When the obstacle and the boundary values are bounded on the boundary, so is the solution to the obstacle problem.

\begin{theorem} \label{thm:boundednessx0}
Suppose that $u \in \mathcal{K}_{g,h}(\Omega,\Omega')$ solves the obstacle problem in $\mathcal{K}_{g,h}(\Omega,\Omega')$.
Let $x_0 \in \partial\Omega$ and suppose that 
\[
\max\bigg\{ \esssup_{B_r(x_0)} g, \esssup_{B_r(x_0) \cap \Omega} h  \bigg\}  < \infty \quad \text{and} \quad \essinf_{B_r(x_0)} g > -\infty
\]
for $r \in (0,r_0)$ with $r_0 := \dist(x_0,\partial \Omega')$. Then $u$ is essentially bounded close to $x_0$. 
\end{theorem}
\begin{proof}
Choose 
\[
k_+ \geq \max\bigg\{ \esssup_{B_r(x_0)} g, \esssup_{B_r(x_0) \cap \Omega} h  \bigg\}  \quad \text{and} \quad
k_- \leq  \essinf_{B_r(x_0)} g.
\]
Then, repeating the proof of \cite[Theorem 1.1]{DKP15} using the estimate \eqref{cacio1} in Lemma \ref{lemma:cacc bnd} with $w_+ := (u-k_+)_+$ and $w_- := (k_--u)_+$, we get
\begin{equation*} 
\esssup_{B_{r/2}(x_0)} w_\pm \leq 
\delta \, {\rm Tail}(w_\pm ; x_0, r/2)+ c\,\delta^{-\gamma} \left( \mean{B_r(x_0)} w_\pm^p \dx \right)^{1/p}.
\end{equation*}
for any $\delta \in (0,1]$ with $\gamma \equiv \gamma(n,p,s)$ and $c\equiv c(n,p,s,\Lambda)$. Consequently, $u$ is essentially bounded in $B_{r/2}(x_0)$.
\end{proof}

To prove the H\"older continuity of the solution to the obstacle problem on the boundary, we also need the following logarithmic estimate. 

\begin{lemma} \label{log lemma w}
Let $B_r \subset B_{R/2}$ be concentric balls and let $w \in W^{s,p}(B_R) \cap L^{p-1}_{sp}(\R^{n})$ satisfy
\[
\esssup_{B_R} w \leq M < \infty \quad \text{and} \quad \essinf_{B_R} w \geq \eps > 0.
\]
Suppose that
\begin{equation*}
\int_{\R^{n}}\int_{\R^{n}}L(w(x),w(y))\left(\frac{M-w(x)}{w(x)^{p-1}}\phi^{p}(x)-\frac{M-w(y)}{w(y)^{p-1}}\phi^{p}(y)\right)K(x,y)\dxy \geq 0,
\end{equation*}
where $\phi \in C^{\infty}_0(B_{3r/2})$ satisfies $0\leq\phi\leq 1$, $\phi = 1$ in $B_{r}$ and $|\nabla\phi|<c/r$.
Then
\begin{align} \label{log lemma claim w}
&\int_{B_{r}}\int_{B_{r}}\left|\log\frac{w(x)}{w(y)}\right|^{p}K(x,y)\dxy \nonumber \\
&\qquad \leq c\,r^{n-sp}\left(1+\eps^{1-p}\left(\frac rR\right)^{sp}{\rm Tail}(w_-,x_0,R)^{p-1}\right).
\end{align}
\end{lemma}
\begin{proof}
By splitting the integral, we obtain
\begin{align} \label{log lemma I1I2}
0 &\leq \int_{B_{2r}}\int_{B_{2r}}L(w(x),w(y))\left(\frac{M-w(x)}{w(x)^{p-1}}\phi^{p}(x)-\frac{M-w(y)}{w(y)^{p-1}}\phi^{p}(y)\right)K(x,y)\dxy \nonumber \\*
&\quad + 2\int_{\R^{n}\setminus B_{2r}}\int_{B_{2r}}L(w(x),w(y))\frac{M-w(x)}{w(x)^{p-1}}\phi^{p}(x)K(x,y)\dxy  \nonumber \\*[1ex]
&=: I_1 + I_2.
\end{align}
We estimate the integrand of $I_1$ in the case $w(x)>w(y)$. By \cite[Lemma 1.3]{DKP15} we have
\[
\phi^{p}(x) \leq \phi^{p}(y)+c\,\delta\phi^{p}(y)+c\,\delta^{1-p}|\phi(x)-\phi(y)|^{p}
\]
whenever $\delta \in (0,1)$. Choosing
\[
\delta=\sigma\frac{w(x)-w(y)}{w(x)} \in (0,1), \quad \sigma \in (0,1),
\]
in the display above, implies
\begin{align*}
\Psi(x,y)&:=(w(x)-w(y))^{p-1}\left(\frac{M-w(x)}{w(x)^{p-1}}\phi^{p}(x)-\frac{M-w(y)}{w(y)^{p-1}}\phi^{p}(y)\right) \\[1ex]
&\leq (w(x)-w(y))^{p-1}\left(\frac{M-w(x)}{w(x)^{p-1}}-\frac{M-w(y)}{w(y)^{p-1}}+c\,\delta\frac{M-w(x)}{w(x)^{p-1}}\right)\phi^{p}(y) \\
&\quad + c\,\delta^{1-p}(w(x)-w(y))^{p-1}\frac{M-w(x)}{w(x)^{p-1}}|\phi(x)-\phi(y)|^{p} \\[1ex]
&= \left(\frac{M-w(x)}{w(x)^{p-1}}-\frac{M-w(y)}{w(y)^{p-1}}+c\,\sigma\frac{(w(x)-w(y))(M-w(x))}{w(x)^{p}}\right) \\
&\qquad \times (w(x)-w(y))^{p-1}\phi^{p}(y) + c\,\sigma^{1-p}(M-w(x))|\phi(x)-\phi(y)|^{p} \\[1ex]
&=: \Psi_1(x,y) + \Psi_2(x,y).
\end{align*}

We estimate $\Psi_1$ separately in the cases $w(x)>2w(y)$ and $w(x) \leq 2w(y)$.
When $w(x)>2w(y)$, we obtain
\begin{align*}
\Psi_1(x,y) &\leq \left(\frac{w(x)-w(y)}{w(y)}\right)^{p-1}\left(2^{1-p}(M-w(x))-(M-w(y))+c\,\sigma M\right)\phi^{p}(y) \\[1ex]
&\leq \left(\frac{w(x)-w(y)}{w(y)}\right)^{p-1}\left((2^{-1}-2^{1-p})w(x)-(1-2^{1-p})M+c\,\sigma M\right)\phi^{p}(y).
\end{align*}
If $p \geq 2$, then
\begin{align*}
(2^{-1}-2^{1-p})w(x)-(1-2^{1-p})M \leq (2^{-1}-2^{1-p})M-(1-2^{1-p})M = -\frac12 M.
\end{align*}
If $1<p<2$, in turn, then
\begin{align*}
(2^{-1}-2^{1-p})w(x)-(1-2^{1-p})M \leq -(1-2^{1-p})M.
\end{align*}
Thus, choosing $\sigma$ small enough in $\Psi_1$, we obtain
\begin{align} \label{psi1wx>>wy}
\Psi_1(x,y) &\leq -\frac1c M\left(\frac{w(x)-w(y)}{w(y)}\right)^{p-1}\phi^{p}(y).
\end{align}

When $w(x) \leq 2w(y)$, we can estimate
\begin{align*}
\Psi_1(x,y) &\leq \left(\frac{w(x)\left((M-w(x))w(y)^{p-1}-(M-w(y))w(x)^{p-1}\right)}{w(y)^{p-1}(w(x)-w(y))}+c\,\sigma M\right) \\*
&\qquad \times \left(\frac{w(x)-w(y)}{w(x)}\right)^{p}\phi^{p}(y).
\end{align*}
If $w(x)<M/2$, then
\begin{align*}
&(M-w(x))w(y)^{p-1}-(M-w(y))w(x)^{p-1} \\[1ex]
&\qquad \qquad\qquad\qquad \leq (M-w(x))w(y)^{p-1}-(M-w(x))w(x)^{p-1} \\[1ex]
&\qquad\qquad\qquad\qquad \leq -\frac1c (M-w(x))(w(x)-w(y))w(x)^{p-2} \\[1ex]
&\qquad \qquad\qquad\qquad\leq -\frac1c M(w(x)-w(y))w(x)^{p-2},
\end{align*}
and we obtain
\begin{align} \label{psi1wx>wy}
\Psi_1(x,y) &\leq -\frac1c M \left(\frac{w(x)-w(y)}{w(x)}\right)^{p}\phi^{p}(y)
\end{align}
when choosing $\sigma$ small enough. If $w(x) \geq M/2$, in turn, then
\begin{align*}
&w(x)\left((M-w(x))w(y)^{p-1}-(M-w(y))w(x)^{p-1}\right) \\[1ex]
&\qquad\qquad\qquad \leq w(x)\left((M-w(x))w(y)^{p-1}-(M-w(y))w(y)^{p-1}\right) \\[1ex]
&\qquad\qquad\qquad \leq -\frac12 M(w(x)-w(y))w(y)^{p-1},
\end{align*}
and again we obtain \eqref{psi1wx>wy} when choosing $\sigma$ small enough.

Let us then estimate further to get logarithms visible. In the case $w(x)>2w(y)$, it holds
\begin{align} \label{logwx>>wy}
\left(\log\frac{w(x)}{w(y)}\right)^{p} \leq c\left(\frac{w(x)-w(y)}{w(y)}\right)^{p-1}
\end{align}
since $(\log t)^{p} \leq c\,(t-1)^{p-1}$ when $t>2$. In the case $w(x) \leq 2w(y)$, in turn, it holds
\begin{align} \label{logwx>wy}
\nonumber \left(\log\frac{w(x)}{w(y)}\right)^{p} &= \left(\log\left(1+\frac{w(x)-w(y)}{w(y)}\right)\right)^{p} \\*[1ex]
&\leq \left(\frac{w(x)-w(y)}{w(y)}\right)^{p}
\ \leq \ 2^{p}\left(\frac{w(x)-w(y)}{w(x)}\right)^{p}
\end{align}
since $\log(1+t) \leq t$ when $t \geq 0$. Thus, combining \eqref{psi1wx>>wy} with \eqref{logwx>>wy} and \eqref{psi1wx>wy} with \eqref{logwx>wy}, we obtain
\begin{align}
\Psi_1(x,y) &\leq -\frac1c M \left(\log\frac{w(x)}{w(y)}\right)^{p}\phi^{p}(y).
\end{align}

For $\Psi_2$ we easily get
\begin{align*}
\Psi_2(x,y) &\leq c\,M |\phi(x)-\phi(y)|^{p} \leq c\,M r^{-p}|x-y|^{p}.
\end{align*}
In the case $w(x)<w(y)$ we can interchange the roles of $x$ and $y$, whereas the contribution from the case $w(x)=w(y)$ is zero.
After integrating we have
\begin{align} \label{log lemma I1}
I_1 \leq -\frac1c M \int_{B_{r}}\int_{B_{r}} \left|\log\frac{w(x)}{w(y)}\right|^{p}K(x,y)\dxy + c\,M r^{n-sp}.
\end{align}

In the integrand of $I_2$ we can estimate
\begin{align*}
&|w(x)-w(y)|^{p-2}(w(x)-w(y))\frac{M-w(x)}{w(x)^{p-1}}\phi^{p}(x) \\[1ex]
&\qquad \leq c\,\big(w(x)^{p-1}+w_-(y)^{p-1}\big)\frac{M}{w(x)^{p-1}}\chi_{B_{3r/2}}(x) \\[1ex]
&\qquad \leq c\,M\left(1+\eps^{1-p}w_-(y)^{p-1} \right)\chi_{B_{3r/2}}(x),
\end{align*}
and integrating yields
\begin{align} \label{log lemma I2}
I_2 &\leq c\,M \int_{\R^{n}\setminus B_{2r}} \int_{B_{3r/2}}\left(1+\eps^{1-p}w_-(y)^{p-1}\right)K(x,y)\dxy \nonumber \\[1ex]
&\leq c\,M \int_{\R^{n}\setminus B_{2r}} \int_{B_{3r/2}}\left(1+\eps^{1-p}w_-(y)^{p-1}\right)|y-x_0|^{-n-sp}\dxy \nonumber \\[1ex]
&\leq c\,M r^{n-sp}\left(1+\eps^{1-p}\,{\rm Tail}(w_-,x_0,r)^{p-1}\right) \nonumber \\[1ex]
&= c\,M r^{n-sp}\left(1+\eps^{1-p}\left(\frac rR\right)^{sp}{\rm Tail}(w_-,x_0,R)^{p-1}\right)
\end{align}
since $w_-=0$ in $B_R$.
Finally, the claim \eqref{log lemma claim w} follows by combining \eqref{log lemma I1I2}, \eqref{log lemma I1} and \eqref{log lemma I2}.
\end{proof}

We employ the previous lemma for particular truncations of the solutions to the obstacle problem. 

\begin{lemma} \label{log lemma u}
Suppose that $u \in \mathcal K_{g,h}(\Omega,\Omega')$ solves the obstacle problem in $\mathcal K_{g,h}(\Omega,\Omega')$. 
Let $B_R \Subset \Omega'$, let $B_r \subset B_{R/2}$ be concentric balls and let
\[
\infty > k_+ \geq \max\bigg\{\esssup_{B_R} g, \esssup_{B_R \cap \Omega} h\bigg\} \quad \text{and} \quad -\infty < k_- \leq \essinf_{B_R} g.
\]
Then the functions
\[
w_\pm:=\esssup_{B_R}(u-k_\pm)_\pm-(u-k_\pm)_\pm+\eps
\]
satisfy the following estimate
\begin{align} \label{log lemma claim u}
&\int_{B_{r}}\int_{B_{r}}\left|\log\frac{w_\pm(x)}{w_\pm(y)}\right|^{p}K(x,y)\dxy \nonumber \\[1ex]
&\qquad \leq c\,r^{n-sp}\left(1+\eps^{1-p}\left(\frac rR\right)^{sp}{\rm Tail}((w_\pm)_-,x_0,R)^{p-1}\right)
\end{align}
for every $\eps>0$.
\end{lemma}
\begin{proof}
Let $\eps>0$ and denote $H_\pm := \esssup_{B_R}(u-k_\pm)_\pm+\eps$. Notice that $H_\pm$ is finite by Theorem \ref{thm:boundednessx0}.
Let $\phi \in C^{\infty}_0(B_{3r/2})$ be such that $0\leq\phi\leq 1$, $\phi \equiv 1$ in $B_{r}$ and $|D\phi|<c/r$.
Denoting by
\[
v_\pm = u \mp \eps^{p-1} \frac{(u-k_\pm)_\pm}{(H_\pm - (u-k_\pm)_\pm)^{p-1}}\phi^{p}
= u \mp \eps^{p-1} \frac{H_\pm-w_\pm}{w_\pm^{p-1}}\phi^{p},
\]
we clearly have $v_\pm \in \mathcal K_{g,h}(\Omega,\Omega')$ since, in particular,
\[
v_+ =  u - \eps^{p-1} \frac{(u-k_+)_+}{w_+^{p-1}}\phi^{p} \geq u - (u-k_+)_+ \chi_{B_R} \geq h
\]
almost everywhere in $\Omega$ because $k_+ \geq  \esssup_{B_R \cap \Omega} h$.

Since $u$ solves the obstacle problem, we have $\langle \mathcal Au, v_\pm-u \rangle \geq 0$, and thus
\begin{align*}
&\mp\int_{\R^{n}}\int_{\R^{n}}L(u(x),u(y))\left(\frac{H_\pm-w_\pm(x)}{w_\pm(x)^{p-1}}\phi^{p}(x)-\frac{H_\pm-w_\pm(y)}{w_\pm(y)^{p-1}}\phi^{p}(y)\right) \\
& \qquad \qquad \ \ \times K(x,y)\dxy \ \geq \ 0.
\end{align*}
Let us estimate the integrand above first for $w_+$. If $u(x),u(y)>k_+$, we simply have $-L(u(x),u(y))=L(w_+(x),w_+(y))$, and consequently
\begin{align} \label{-Luxuy<Lwxwy}
&-L(u(x),u(y))\left(\frac{H_+-w_+(x)}{w_+(x)^{p-1}}\phi^{p}(x)-\frac{H_+-w_+(y)}{w_+(y)^{p-1}}\phi^{p}(y)\right) \nonumber \\[1ex]
&\qquad \leq L(w_+(x),w_+(y))\left(\frac{H_+-w_+(x)}{w_+(x)^{p-1}}\phi^{p}(x)-\frac{H_+-w_+(y)}{w_+(y)^{p-1}}\phi^{p}(y)\right)
\end{align}
holds. If $u(x)>k_+\geq u(y)$, then $w_+(y)=H_+$ and
\[
-(u(x)-u(y))=-(H_+ -w_+(x) +k_+-u(y)) \leq w_+(x)-w_+(y),
\]
and \eqref{-Luxuy<Lwxwy} follows. If, in turn, $u(y)>k_+\geq u(x)$, we can just exchange the roles of $x$ and $y$ to obtain \eqref{-Luxuy<Lwxwy}, whereas in the case $u(x),u(y) \leq k_+$ \eqref{-Luxuy<Lwxwy} is trivial since $w_+(x)=w_+(y)=H_+$.
Similarly, we obtain
\begin{align*} \label{Luxuy<Lwxwy}
&L(u(x),u(y))\left(\frac{H_- -w_-(x)}{w_-(x)^{p-1}}\phi^{p}(x)-\frac{H_- -w_-(y)}{w_-(y)^{p-1}}\phi^{p}(y)\right) \nonumber \\*[1ex]
&\qquad \leq L(w_-(x),w_-(y))\left(\frac{H_- -w_-(x)}{w_-(x)^{p-1}}\phi^{p}(x)-\frac{H_- -w_-(y)}{w_-(y)^{p-1}}\phi^{p}(y)\right)
\end{align*}
for $w_-$, and thus
\begin{align*}
& \int_{\R^{n}}\int_{\R^{n}}L(w_\pm(x),w_\pm(y))\left(\frac{H_\pm-w_\pm(x)}{w_\pm(x)^{p-1}}\phi^{p}(x)-\frac{H_\pm-w_\pm(y)}{w_\pm(y)^{p-1}}\phi^{p}(y)\right) \\
& \qquad \qquad \ \ \times K(x,y)\dxy \  \geq\  0.
\end{align*}
Now the claim \eqref{log lemma claim u} follows from Lemma \ref{log lemma w}.
\end{proof}

\subsection{H\"older continuity up to the boundary}
Before starting to prove the H\"older continuity up to the boundary, it is worth noticing that we have to assume $g \in \mathcal{K}_{g,h}(\Omega,\Omega')$
since otherwise the solution may be discontinuous on every boundary point.
\begin{example}
Suppose that $sp<1$ and let $\Omega=B_{1}(0)$ and $\Omega'=B_{2}(0)$. Then the characteristic function $\chi_\Omega$ solves the obstacle
problem in $\mathcal K_{g,h}(\Omega,\Omega')$ with constant functions $g\equiv 0$ and $h \equiv 1$.
Indeed, $\chi_\Omega \in W^{s,p}(\Omega')$ when $sp<1$ and it is easy to see that it is a weak supersolution. Consequently, it is the solution to the obstacle problem in $\mathcal K_{g,h}(\Omega,\Omega')$ in view of Proposition \ref{smallest super}.
\end{example}

In the following we will assume that the complement of $\Omega$ satisfies the following measure density condition.
There exist $\delta_\Omega \in (0,1)$ and $r_0>0$ such that for every $x_0 \in \partial\Omega$
\begin{equation} \label{eq:dens cond}
\inf_{0<r<r_0} \frac{| (\R^n \setminus \Omega) \cap B_r(x_0)|}{|B_r(x_0)|} \geq \delta_\Omega.
\end{equation}
As mentioned in the introduction, this requirement is in the same spirit of the standard Nonlinear Potential Theory, rearranged as an information given only on the complement of the set in accordance with the nonlocality of the involved equations; and this is also an improvement with respect to the previous boundary regularity results in the fractional literature when strong smooth or Lipschitz regularity of the set is required.

\begin{lemma} \label{lemma:dens cond}
Let $\Omega$ satisfy \eqref{eq:dens cond} for $r_0>0$ and $\delta_\Omega>0$,
and let $B\equiv B_r(x_0)$ with $x_0 \in \partial\Omega$ and $r \in (0,r_0)$.
Suppose that $f \in W^{s,p}(B)$ and $f=0$ in $B \setminus \Omega$. Then
\begin{equation}
\mean{B}|f|^{p}\dx \leq c\left(1-(1-\delta_\Omega)^{1-1/p}\right)^{-p}r^{sp}\int_B\mean{B}\frac{|f(x)-f(y)|^{p}}{|x-y|^{n+sp}}\dxy.
\end{equation}
\end{lemma}
\begin{proof}
Since $f=0$ in $B\setminus\Omega$,
\begin{align*}
|f_B| &\leq \frac{|B\cap\Omega|}{|B|}\mean{B\cap\Omega}|f|\dx \leq \frac{|B\cap\Omega|}{|B|}\left(\mean{B\cap\Omega}|f|^{p}\dx\right)^{1/p} \\[1ex]
&\leq\left(\frac{|B\cap\Omega|}{|B|}\right)^{1-1/p}\left(\mean{B}|f|^{p}\dx\right)^{1/p}
\ =  \ (1-\delta_\Omega)^{1-1/p}\left(\mean{B}|f|^{p}\dx\right)^{1/p},
\end{align*}
and we can estimate
\begin{align*}
\left(\mean{B}|f|^{p}\dx\right)^{1/p} &\leq \left(\mean{B}|f-f_B|^{p}\dx\right)^{1/p} + |f_B| \\[1ex]
&\leq \left(\mean{B}|f-f_B|^{p}\dx\right)^{1/p} + (1-\delta_\Omega)^{1-1/p}\left(\mean{B}|f|^{p}\dx\right)^{1/p}.
\end{align*}
Absorbing the last term yields
\begin{align*}
\mean{B}|f|^{p}\dx \leq \left(1-(1-\delta_\Omega)^{1-1/p}\right)^{-p}\mean{B}|f-f_B|^{p}\dx,
\end{align*}
and the claim follows from the fractional Poincar\'e inequality.
\end{proof}

\begin{lemma} \label{lemma:osc reduction}
Assume that $x_0 = 0 \in \partial \Omega$ and $g(0)=0$, where $\Omega$ satisfies \eqref{eq:dens cond} for  all $r\leq R$.  Let $\omega > 0$. There exist $\tau_0 \in (0,1)$, $\sigma \in (0,1)$ and $\theta \in (0,1)$, all depending only on $n$, $p$, $s$ and $\delta_\Omega$, such that
if
\begin{equation} \label{oscB}
\osc_{B_R(0)}u + \sigma{\rm Tail}(u;0,R) \leq \omega \quad \text{and} \quad \osc_{B_R(0)}g \leq \frac\omega 8
\end{equation}
hold, then the decay estimate 
\begin{equation} \label{osctauB}
\osc_{B_{\tau R}(0)}u + \sigma{\rm Tail}(u;0,\tau R) \leq (1-\theta)\omega
\end{equation}
holds as well for every $\tau \in (0,\tau_0]$.
\end{lemma}
\begin{proof}
Denote $H=\theta/\sigma$ and $B\equiv B_R(0)$. We begin by estimating the tail term to obtain
\begin{align*}
\sigma^{p-1}{\rm Tail}(u;0,\tau R)^{p-1} &= \sigma^{p-1}(\tau R)^{sp}\int_{B \setminus \tau B}\frac{|u(x)|^{p-1}}{|x|^{n+sp}}\dx \\
& \quad  + \sigma^{p-1}\tau^{sp}{\rm Tail}(u;0,R)^{p-1} \\[1ex]
&\leq c\,\sigma^{p-1}\omega^{p-1}+\tau^{sp}\omega^{p-1}
\end{align*}
by \eqref{oscB}. Consequently,
\begin{equation} \label{eq:Tail bndr osc 000}  
\sigma{\rm Tail}(u;0,\tau R) \leq \tilde c \left(\frac\theta H + \tau^{sp/(p-1)}\right)\omega
 \leq \frac{2\tilde c\,\theta} H \omega = \theta\omega
\end{equation}
when restricting $\tau_0 \leq \sigma^{(p-1)/(sp)}$ and choosing $H = 2\tilde c \geq 1$, where $\tilde c \equiv \tilde c(n,p,s)$.
Thus, it suffices to prove that
\begin{equation} \label{osctauB2}
\osc_{\tau B}u \leq (1-2\theta)\omega
\end{equation}
for all $\tau \leq \tau_0$. To this end, let
\[
k_+ := \sup_{B} u- \frac{\omega}{4}, \quad k_- := \inf_{B}u + \frac{\omega}{4}  , \quad \eps:=\theta\omega
\]
and
\[
w_\pm:=\sup_{B}(u-k_\pm)_\pm-(u-k_\pm)_\pm+\eps, \quad \tilde w_\pm:=\frac{w_\pm}{\sup_{B} w_\pm}.
\] 

We may assume $\sup_{B}u \geq \frac38\omega$ or $\inf_{B}u \leq -\frac38\omega$ since otherwise
$\osc_{\tau B}u \leq \osc_{B}u \leq \frac34\omega$ and there is nothing to prove if we assume that $\theta \leq 1/8$.
We consider the case $\sup_{B}u \geq \frac38\omega$; the case $\inf_{B}u \leq -\frac38\omega$ is symmetric.
Notice that we have $\tilde w_+ = 1$ in $B\setminus \Omega$ due to the condition $u=g \leq \omega/8$ in $B\setminus \Omega$. First, we estimate, by Lemmas \ref{lemma:dens cond} and \ref{log lemma u} with $r\equiv 2\tau R$ and \eqref{oscB} when restricting $\tau_0 \leq 1/4$ and $\tau_0 \leq \sigma^{2(p-1)/(sp)}$, to obtain 
\begin{align*}
\mean{2\tau B}\left|\log \tilde w_+\right|^{p}\dx &\leq c\,(\tau R)^{sp}\int_{2\tau B}\mean{2\tau B}\left|\log\frac{\tilde w_+(x)}{\tilde w_+(y)}\right|^{p}K(x,y)\dxy \\[1ex]
&\leq c\,(\tau R)^{sp}\int_{2\tau B}\mean{2\tau B}\left|\log\frac{w_+(x)}{w_+(y)}\right|^{p}K(x,y)\dxy \\[1ex]
&\leq c\,\Big(1+(\theta\omega)^{1-p}\tau^{sp}{\rm Tail}((w_+)_-;0,R)^{p-1}\Big) \\[1ex]
&\leq c\left(1+(\theta\omega)^{1-p}\sigma^{2(p-1)}{\rm Tail}(u;0,R)^{p-1}\right) \\[1ex]
&\leq c\left(1+(\theta\omega)^{1-p}\left(\frac\theta H\right)^{p-1}\omega^{p-1}\right) \\[1ex]
&\leq c.
\end{align*}
Consequently, by Chebyshev's Inequality we have
\begin{align} \label{logtildew>M} \nonumber
\frac{\left|2\tau B \cap \{|\log \tilde w_+| \geq  \left|\log(20\,\theta)\right| \}\right|}{|2\tau B|} &\leq  |\log (20\,\theta)|^{-p}\mean{2\tau B}|\log\tilde w_+|^{p}\dx 
\\*[1ex]
 &\leq c\,|\log (20\,\theta)|^{-p}.
\end{align}

Let us estimate the left-hand side of \eqref{logtildew>M}. Since, by definitions, $0 < \tilde w_+ \leq 1$ and $\sup_{B} (u-k_+)_+ = \omega/4$, we have that
\begin{align*}
\left\{|\log \tilde w_+| \geq|\log (20\,\theta)| \right\} &= \left\{\tilde w_+ \leq 20\,\theta\right\} 
\\*[1ex]
 & = \Big\{\frac{\omega}{4} + \eps - (u-k_+)_+ \leq 20\,\theta \Big(\frac{\omega}{4} + \eps\Big) \Big\}
\\*[1ex]
 & = \Big\{\frac{\omega}{4} + \theta\omega - u + \sup_B u - \frac{\omega}{4} \leq 5 \theta \omega + 20\,\theta^2 \omega  \Big\}
\\*[1ex]
 & \supset \Big\{u  \geq \sup_B u - 4 \theta \omega \Big\}
\end{align*}
provided that $\theta < 1/20$. Consequently, by defining $\tilde k \equiv \tilde k_+ := \sup_B u - 4 \theta \omega$ and using the above two displays, we get
\begin{align*}
\left( \mean{2\tau B} (u - \tilde k )_+^p \dx \right)^{1/p} & \leq 4 \theta \omega  \left( \frac{ |2\tau B \cap \{ u  \geq \sup_B u - 4 \theta \omega  \}| }{|2\tau B| }\right)^{1/p} \\[1ex]
&\leq 4 \theta \omega \left( \frac{ |2\tau B \cap \{ |\log \tilde w_+|  \geq|\log (20\,\theta)|  \}| }{|2\tau B| }\right)^{1/p} \\[1ex]
&\leq \frac{c\,\theta\omega}{|\log(20\,\theta)|}.
\end{align*}

Since $\tilde k \geq \sup_B g$, we have by Theorem \ref{thm:boundednessx0} that
\begin{equation*} 
\sup_{\tau B} (u - \tilde k )_+ \leq 
\delta \, {\rm Tail}((u - \tilde k )_+ ; 0, \tau R)+ c\,\delta^{-\gamma} \left( \mean{2\tau B} (u - \tilde k )_+^p \dx \right)^{1/p}
\end{equation*}
for any $\delta \in (0,1]$, and hence
\begin{equation} \label{suputauB0}
\sup_{\tau B} u  \leq \sup_B u - 4 \theta\omega + \delta \, {\rm Tail}((u - \tilde k )_+ ; 0, \tau R) + \frac{c\,\delta^{-\gamma }}{|\log(20\,\theta)|} \theta\omega.
\end{equation}
To estimate the tail term, we proceed similarly as in~\eqref{eq:Tail bndr osc 000} and obtain
\begin{align*}
{\rm Tail}((u - \tilde k )_+; 0,\tau R)^{p-1}  &\leq (\tau R)^{sp} \int_{B \setminus \tau B} (u(x) - \tilde k )_+^{p-1} |x|^{-n-sp}\dx \\
&\quad + \tau^{sp} {\rm Tail}(u;0,R)^{p-1} \\[1ex]
&\leq c\,(\theta\omega)^{p-1} \left( 1 + \frac{\tau^{sp}}{\theta^{p-1}\sigma^{p-1}}  \right) 
 \\*[1ex]
  &\leq c\,(\theta\omega)^{p-1},
\end{align*}
where we also used the facts $(u-\tilde k)_+ \leq 4\theta\omega$ in $B$, ${\rm Tail}(u;0,R) \leq \omega/\sigma$ by \eqref{oscB}, and  $\tau^{sp} \leq \tau_0^{sp} \leq \theta^{p-1}\sigma^{p-1}$. Thus, by choosing first $\delta$ small and then $\theta$ small accordingly, we deduce from \eqref{suputauB0} that 
\[
\sup_{\tau B} u  \leq \sup_B u - 2 \theta\omega,
\]
and \eqref{osctauB2} follows, as desired. This finishes the proof.
\end{proof}
 
Now, we have finally collected all the machinery to plainly deduce the H\"older continuity up the boundary. We have the following
\begin{theorem} \label{thm:H cont bdry}
Suppose that $u$ solves the obstacle problem in $\mathcal K_{g,h}(\Omega,\Omega')$ and assume $x_0\in \partial\Omega$ and $B_{2R}(x_0) \subset \Omega'$.
If $g \in \mathcal K_{g,h}(\Omega,\Omega')$ is H\"older continuous in $B_R(x_0)$ and $\Omega$ satisfies \eqref{eq:dens cond} for all $r\leq R$, then $u$ is H\"older continuous in $B_R(x_0)$ as well.
\end{theorem}
\begin{proof}
We may assume $x_0=0$ and $g(0)=0$.
Moreover, we may choose $R_0$ such that $\osc_{B_0}g \leq \osc_{B_0}u$ for $B_0 \equiv B_{R_0}(0)$ since otherwise we have nothing to prove, and define
\begin{equation} \label{omega0}
\omega_0 := 8\left(\osc_{B_0}u + {\rm Tail}(u;0,R_0)\right).
\end{equation}
By Lemma \ref{lemma:osc reduction} there exist $\tau_0$, $\sigma$ and $\theta$ depending only on $n$, $p$, $s$ and $\delta_\Omega$ such that if
\begin{equation} \label{oscB0}
\osc_{B_r(0)}u + \sigma{\rm Tail}(u;0,r) \leq \omega \quad \text{and} \quad \osc_{B_r(0)}g \leq \frac\omega 8
\end{equation}
hold for a ball $B_r(0)$ and for $\omega>0$, then
\begin{equation} \label{osctauB0}
\osc_{B_{\tau r}(0)}u + \sigma{\rm Tail}(u;0,\tau r) \leq (1-\theta)\omega
\end{equation}
holds for every $\tau \in (0,\tau_0]$.  As we can take $\tau \leq \tau_0$ such that
\begin{equation} \label{osctauBg}
\osc_{\tau^{j}B_0}g \leq (1-\theta)^{j}\frac{\omega_0}{8} \qquad \text{for every } j=0,1,\dots.
\end{equation}
Now, iterating \eqref{osctauB0} with \eqref{oscB0} and \eqref{osctauBg} noticing also that the initial condition is satisfied by \eqref{omega0}, we obtain
\begin{equation*} \label{osctauBu}
\osc_{\tau^{j}B_0}u \leq (1-\theta)^{j}\omega_0 \qquad \text{for every } j=0,1,\dots.
\end{equation*}
Consequently, $u \in C^{0,\alpha}(B_0)$ with the exponent $\alpha=\log(1-\theta)/\log \tau \in (0,1)$.
\end{proof}

Slightly modifying the proof above, we easily obtain the following.
\begin{theorem} \label{thm:cont bdry}
Suppose that $u$ solves the obstacle problem in $\mathcal K_{g,h}(\Omega,\Omega')$ and assume $x_0\in \partial\Omega$ and $B_{2R}(x_0) \subset \Omega'$.
If $g \in \mathcal K_{g,h}(\Omega,\Omega')$ is continuous in $B_R(x_0)$ and $\Omega$ satisfies \eqref{eq:dens cond} for all $r\leq R$, then $u$ is continuous in $B_R(x_0)$ as well.
\end{theorem}

For the sake of completeness, we gather our continuity results into two global theorems. The first one follows by combining Theorems \ref{thm:obs H cont} and \ref{thm:H cont bdry} and the second one by combining Theorems \ref{thm:obs cont} and \ref{thm:cont bdry}.
\begin{theorem} \label{thm:H cont up to bdry}
Suppose that $\Omega$ satisfies \eqref{eq:dens cond}  
and $g \in \mathcal K_{g,h}(\Omega,\Omega')$. Let $u$ solve the obstacle problem in $\mathcal K_{g,h}(\Omega,\Omega')$. 
If $g$ is locally H\"older continuous in $\Omega'$ and $h$ is locally H\"older continuous in $\Omega$, then $u$ is locally H\"older continuous in $\Omega'$.
\end{theorem}

\begin{theorem} \label{thm:cont up to bdry}
Suppose that $\Omega$ satisfies \eqref{eq:dens cond}  
and $g \in \mathcal K_{g,h}(\Omega,\Omega')$. Let $u$ solve the obstacle problem in $\mathcal K_{g,h}(\Omega,\Omega')$.
If $g$ is continuous in $\Omega'$ and $h$ is continuous in $\Omega$, then $u$ is continuous in $\Omega'$.
\end{theorem}

\end{document}